\documentclass[12pt,letterpaper]{amsart}
\usepackage[pass]{geometry}
\usepackage{amssymb,amsmath,amsthm,amsgen}

\usepackage{tikz}
\usepackage{comment}
\usepackage{}
\usepackage{enumitem}

\newcommand{\old}[1]{}
\theoremstyle{plain}
\newtheorem{thm}{Theorem}[section]
\newtheorem{lem}[thm]{Lemma}
\newtheorem{conj}{Conjecture}
\newtheorem{cor}[thm]{Corollary}

\newtheorem{prop}[thm]{Proposition}
\theoremstyle{definition}
\newtheorem{defn}[thm]{Definition}

\newtheorem{ex}[thm]{Example}

\newtheorem{qn}[thm]{Question}

\numberwithin{equation}{section}

 at 10truept

\title{On a curious variant of the $S_n$-module $Lie_n$}

\author{Sheila Sundaram}
\address{Pierrepont School, One Sylvan Road North, Westport, CT 06880}
\email{shsund@comcast.net}

\date{\today}


\begin{document}
\begin{abstract} We introduce a variant of the much-studied $Lie$ representation of the symmetric group $S_n$, which we denote by $Lie_n^{(2)}.$ Our variant gives rise to a decomposition of the regular representation  as a sum of {exterior} powers of modules $Lie_n^{(2)}.$ This is in contrast to the theorems of Poincar\'e-Birkhoff-Witt and Thrall which decompose the regular representation into a sum of symmetrised $Lie$ modules. We show that nearly every known property of $Lie_n$ has a counterpart for the module $Lie_n^{(2)},$ suggesting connections to the cohomology of configuration spaces via the character formulas of Sundaram and Welker, to the Eulerian idempotents of Gerstenhaber and Schack, and to the Hodge decomposition of the complex of injective words arising from  Hochschild homology, due to  Hanlon and Hersh.

\end{abstract}
\maketitle

\section{Introduction}
In this paper we present 
 the unexpected discovery, announced in \cite{SuFPSAC2018}, 
 of a curious variant of  the $S_n$-module    
 $Lie_n$ afforded by the multilinear component of the free Lie algebra with $n$ generators. The theorems of Poincar\'e-Birkhoff-Witt  and Thrall (see, e.g. \cite{R}) state that the universal enveloping algebra of the free Lie algebra is the symmetric algebra over the free Lie algebra, and hence coincides with the full tensor algebra. This is equivalent, via Schur-Weyl duality, to Thrall's decomposition of the regular representation into a sum of symmetric powers of the representations $Lie_n.$ By contrast, here we obtain a decomposition of the regular representation as a sum of exterior powers of modules (Theorem \ref{ExtReg}).  The key ingredient is our variant of $Lie_n,$ an $S_n$-module that we denote by $Lie_n^{(2)},$ which turns out to possess remarkable properties akin to those of  $Lie_n.$  Our results (see Theorems \ref{Compare},  
 \ref{LehrerHLLieSup2}, \ref{EquivLieSup2}) 
 bear a striking resemblance to properties of the Whitney homology of the partition lattice (and hence the Orlik-Solomon algebra for the root system $A_n$),  and to the computation of the cohomology of  the configuration space for the braid arrangement found in \cite[Theorem 4.4]{SW}.
 In particular these properties indicate the possibility of an 
underlying algebra structure for $Lie_n^{(2)}$ involving an acyclic complex. Theorem \ref{EquivLieSup2} furthers this analogy;  we show that $Lie_n^{(2)}$ admits a filtration close to the one arising from the derived series of the free Lie algebra.  There is an interesting  action on derangements  arising from $Lie_n^{(2)}$ as well (Theorem \ref{HodgeFilt}); we prove that $Lie_n^{(2)}$ gives rise to a new decomposition of the homology of the complex of injective words studied by Reiner and Webb \cite{RW}, one that is different from the Hodge decomposition of Hanlon and Hersh \cite{HH}. These results are collected in Section 2,  showing that for every well-known property of  $Lie_n,$  the representation $Lie_n^{(2)}$  offers an interesting counterpart.

A characteristic feature of the  complement of the $A_{n-1}$-hyperplane arrangement in complex space, and hence the configuration spaces associated to the braid arrangement, is that the cohomology ring has the structure of a symmetric or exterior algebra over the top  cohomology as an $S_n$-module  (see Theorem \ref{SWThm4.4} and Theorem \ref{Compare}).   Moreover this top cohomology is  $Lie_n$ or its sign-tensored version, and thus its character values are supported on a specific class of permutations: those whose cycles all have the same length.  

The higher Lie modules, first defined in \cite{T}, figure prominently in all the situations mentioned above, and hence the language of symmetric functions and plethysm is ideal for describing the results.  This is  precisely the framework of the symmetric function identities  developed in \cite{Su1};  these crucial identities are described in Section~\ref{SecMetaThms}, where we  state the key result from \cite{Su1}, Theorem~\ref{metathm}, and compile a toolkit  that has proved useful in manipulating plethysms arising from homology representations.  One interesting consequence  is a fact that does not appear to have been previously observed, namely the \textit{equivalence} of all  the known representation-theoretic properties of $Lie$ (the formulas of Thrall and Cadogan, the filtration arising from the derived series, the appearance of the $Lie$ character in the action on derangements).   This is explained in Theorem \ref{EquivPBW}. 

The module $Lie_n^{(2)}$ is a special case of a family of variations of $Lie_n,$ whose discovery arose from  the investigation begun in \cite{Su1} on the positivity of row sums in the character table of $S_n$. 
Indeed, the symmetrised powers of $Lie_n^{(2)}$ itself give the representation obtained by taking row sums  for the  subset of conjugacy classes corresponding to cycles whose lengths are a power of 2 (Theorem~\ref{SymExt}). The more general results were announced in \cite{SuFPSAC2018} and \cite{SuFPSAC2019}, and will be the subject of  a separate paper.

\subsection{Preliminaries}
We follow \cite{M} and \cite{St4EC2} for notation regarding symmetric functions.  In particular, $h_n,$ $e_n$ and $p_n$ denote respectively the complete homogeneous, elementary and power sum symmetric functions.  If $\mathrm{ch}$ is the Frobenius characteristic map from the representation ring of the symmetric group $S_n$ to the ring of symmetric functions with real coefficients, then 
$h_n=\mathrm{ch}(1_{S_n})$ is the characteristic of the trivial representation, and $e_n=\mathrm{ch}({\rm sgn}_{S_n})$ is the characteristic of the sign representation of $S_n.$   If $\mu$ is a partition of $n$ then 
define $p_\mu =\prod_i p_{\mu_i};$ $h_\mu$ and $e_\mu$ are defined  in similar multiplicative fashion.  
The Schur function $s_\mu$  indexed by the partition $\mu$ is the Frobenius characteristic of the $S_n$-irreducible indexed by $\mu.$  
  Finally,  the involution $\omega$ 
   takes $h_n$ to $e_n$ corresponding to tensoring with the sign  representation.

If $q$ and $r$ are characteristics of representations of $S_m$ and $S_n$ respectively, they yield a representation of the wreath product $S_m[S_n]$ in a natural way, with the property that when this representation is induced up to $S_{mn},$ its Frobenius characteristic is the plethysm $q[r].$ For more background about this operation, see \cite{M}.  We will make extensive use of the properties of this operation, in particular the fact that plethysm with a symmetric function $r$ is an endomorphism on the ring of symmetric functions \cite[(8.3)]{M}.  See also \cite[Chapter 7, Appendix 2, A2.6]{St4EC2}.

Define \begin{align}\label{defHE} &H(t)=\sum_{i\geq 0}t^i  h_i=\exp \sum_{i\ge 1}\frac{t^ip_i}{i}, \quad E(t) = \sum_{i\geq 0} t^i  e_i=\exp \sum_{i\ge 1}(-1)^{i-1}\frac{t^ip_i}{i}; \\
&H=\sum_{i\geq 0}  h_i, \quad E= \sum_{i\geq 0}  e_i; \quad H^{\pm}=\sum_{r\geq 0} (-1)^r h_r, \quad E^{\pm}=\sum_{r\geq 0} (-1)^r e_r. 
\end{align} 
Now let $\{q_i\}_{i\geq 1}$ be a sequence of symmetric functions, each $q_i$ homogeneous of degree $i.$  
Let $Q=\sum_{i\geq 1}q_i,$  $Q(t)=\sum_{n\geq 1} t^n q_n.$
For each partition $\lambda$ of $n\geq 1$  with $m_i(\lambda)=m_i$ parts equal to $i\geq 1,$ let $|\lambda|=n = \sum_{i\geq 1} i m_i$ be the size of $\lambda,$ and 
$\ell(\lambda)=\sum_{i\geq 1} m_i(\lambda)=\sum_{i\geq 1} m_i$ be the length (total number of parts) of $\lambda.$ 

Define 
\begin{equation}\label{HigherQ}
 H_\lambda[Q]=\prod_{i:m_i(\lambda)\geq 1} h_{m_i}[q_i],\qquad  \qquad E_\lambda[Q]=\prod_{i:m_i(\lambda)\geq 1} e_{m_i}[q_i].
 \end{equation}

For the empty partition (of zero) we define $H_\emptyset [Q]=1=
E_\emptyset[Q]=H^{\pm}_\emptyset [Q]=
E^{\pm}_\emptyset[Q].$

 Consider the generating functions 
$H[Q](t)$ and $E[Q](t).$ 
With the convention that $Par$, the set of all partitions of nonnegative integers,
 includes the unique empty partition of zero,  by the preceding observations and standard properties of plethysm \cite{M} we have 
 
 \begin{equation}\label{defhrofQ}h_r[Q]|_{{\rm deg\ }n}=\sum_{\stackrel {\lambda\vdash n}{\ell(\lambda)=r}} H_\lambda[Q], 
 \qquad \text{and } \quad 
 e_r[Q]|_{{\rm deg\ }n}=\sum_{\stackrel{\lambda\vdash n}{\ell(\lambda)=r}}E_\lambda[Q];
 \end{equation}

$$H[Q](t)=\sum_{\lambda\in Par} t^{|\lambda|} H_\lambda[Q], \qquad \text{and } \quad E[Q](t)=\sum_{\lambda\in Par} t^{|\lambda|} E_\lambda[Q].$$

Also write $Q^{alt}(t)$ for the alternating sum $\sum_{n\geq 1} (-1)^{i-1} t^i q_i = t q_1-t^2 q_2+t^3 q_3-\ldots.$  

Let $\psi(n)$ be any real-valued function defined on the positive integers. 
Define symmetric functions $f_n$ by 
 \begin{equation}\label{definef_n}f_n = \dfrac{1}{n} \sum_{d|n} \psi(d) p_d^{\frac{n}{d}},
\text{ so that }
\omega(f_n) =  \dfrac{1}{n} \sum_{d|n} \psi(d) (-1)^{n-\frac{n}{d}} p_d^{\frac{n}{d}}.\end{equation}
Note that,  when $\psi(1)$ is a positive integer,  this makes $f_n$ the Frobenius characteristic of a possibly virtual $S_n$-module whose dimension is $(n-1)!\psi(1).$

Finally, define the associated polynomial in one variable, $t,$ by
\begin{equation}\label{definepolyf_n}
f_n(t) =\dfrac{1}{n} \sum_{d|n} \psi(d) t^{\frac{n}{d}}.
\end{equation}

\section{A comparison of $Lie_n$ and the variant $Lie_n^{(2)}$}

In this section we define  the $S_n$-module $Lie_n^{(2)}$ and describe some of its remarkable properties.  The goal here is to  analyse this module by  interpreting the plethystic identities it satisfies in an interesting representation-theoretic and homological context.  The  properties are established using plethystic symmetric function techniques applied to the Frobenius characteristic of $Lie_n^{(2)}$;  we have relegated the technical details of the proofs to the next section.        
 The present section has been written to be self-contained.

Recall \cite{R} that the $S_n$-module $Lie_n$ is the action of $S_n$ on the multilinear component of the free Lie algebra, and coincides with the  induced representation $\exp(\frac{2i\pi}{n})\uparrow_{C_n}^{S_n},$ where $C_n$ is the cyclic group generated by an $n$-cycle in $S_n$.  Its Frobenius characteristic is obtained by taking $\psi(d)=\mu(d)$ (the number-theoretic M\"obius function) in equation~(\ref{definef_n}).

Another module that will be of interest is the $S_n$-module $C\!onj_n$ afforded by the conjugacy action of $S_n$ on the class of $n$-cycles.  Clearly we have $C\!onj_n\simeq {\mathbf 1}\uparrow_{C_n}^{S_n}.$  Its Frobenius characteristic is obtained by taking $\psi(d)=\phi(d)$ (Euler's totient function) in equation~(\ref{definef_n}).

\begin{defn}\label{defLieSup2} Let $k_n$ be the highest power of 2 dividing $n.$ Define $Lie_n^{(2)}$ to be the induced module 
$$\exp(\tfrac{2i\pi}{n}\cdot 2^{k_n})\uparrow_{C_n}^{S_n}.$$
\end{defn}

The first two of the following facts are now immediate.  $Lie_n^{(2)}$ is $S_n$-isomorphic to 
\begin{itemize} 
\item   $ Lie_n$ if $n$ is odd;
\item $C\!onj_n$ if $n$ is a power of 2;
\item $Lie_n\otimes {\bf sgn}_{S_n}$ if $n$ is twice an odd number.
\end{itemize}

The third fact follows, for example, by first establishing the isomorphism 
$${\bf sgn}_{S_n}\otimes \chi\uparrow_{C_n}^{S_n}\simeq ({\bf sgn}^{n-1}_{C_n}\otimes \chi)\uparrow_{C_n}^{S_n}.$$
A different proof  is given in Theorem~\ref{LietoLiesup2}.

   Since  it is most convenient to use the language of symmetric functions,  we will often abuse notation and use $Lie_n$ and $Lie_n^{(2)}$ to mean both the module and its Frobenius characteristic.
   
 We write $Lie$ for the sum of symmetric functions $\sum_{n\geq 1} Lie_n$ and $Lie^{(2)}$ for the sum 
 $\sum_{n\geq 1} Lie_n^{(2)}.$
Recall from Section 1.1 that we define,  for each partition $\lambda$ of $n\geq 1$  with $m_i(\lambda)=m_i$ parts equal to $i,$
 $H_\lambda[Q]=\prod_{i:m_i(\lambda)\geq 1} h_{m_i}[q_i]$ and $E_\lambda[Q]=\prod_{i:m_i(\lambda)\geq 1} e_{m_i}[q_i];$  see also equation (\ref{defhrofQ}).
 Finally, recall that $p_1^n=h_1^n=e_1^n$ is the Frobenius characteristic of the regular representation ${\mathbf 1}\uparrow_{S_1}^{S_n}$ of $S_n.$
 
 The $H_\lambda[Lie]$ are the (Frobenius characteristics of) the higher Lie modules appearing in Thrall's decomposition of the regular representation (see below  for more details).  
 We denote the wreath product of $S_a$ with $a$ copies of $S_b$ by $S_a[S_b];$ explicitly it is the normaliser of the direct product 
 $\underbrace{S_b\times \ldots \times S_b}_{a}$  in $S_{ab}.$ 
 Given representations $V_a$ and $V_b$  of $S_a, S_b,$ respectively, 
there is an obvious associated representation $V_a[V_b]$ of the wreath product $S_a[S_b]$, whose Frobenius characteristic is given by the plethysm ${\rm ch\,}V_a[{\rm ch\,}V_b].$  The higher Lie module 
$H_\lambda[Lie],$ for a partition $\lambda$ of $n$ with 
$m_i$ parts equal to $i,$  is the characteristic of  the induced representation 
\begin{equation*}{\Large\otimes}_i {\bf 1}_{S_{m_i}}[Lie_i]
{\large\uparrow}_{\prod_i S_{m_i}[S_i]}^{S_n}.\end{equation*}

If $X$ is any topological space, then the ordered configuration space $C\!on\! f_n \, X$ of $n$ distinct points in $X$ is defined to be the set 
$\{(x_1,\ldots, x_n): i\neq j \Longrightarrow x_i\neq x_j\}.$ The symmetric group $S_n$ acts on $C\!on\! f_n\, X$ by permuting coordinates, and hence induces an action on the cohomology $H^k(C\!on\! f_n\,  X, {\mathbb Q}), k\geq 0.$  

\begin{thm}\label{SWThm4.4}\cite[Theorem 4.4, Corollary 4.5]{SW} For all $d\geq 1, $  and $0\leq k\leq n-1,$ the Frobenius characteristic of 
\begin{itemize}
\item 
$H^k( C\!on\! f_n\, {\mathbb R}^{2},{\mathbb Q} ) \simeq H^{(2d-1)k}(C\!on\! f_n\, {\mathbb R}^{2d},{\mathbb Q})$ is   
$\omega\left(e_{n-k}[Lie]|_{\text{deg }n} \right).$
\item
$H^{2k}( C\!on\! f_n\, {\mathbb R}^{3} ,{\mathbb Q}  ) \simeq H^{2dk}(C\!on\! f_n\, {\mathbb R}^{2d+1}, {\mathbb Q})$   is 
$
h_{n-k}[Lie]|_{\text{deg }n} .$
\end{itemize}
The cohomology vanishes in all other degrees.

When $d=1,$ $H^{0}( C\!on\! f_n\, {\mathbb R} ,{\mathbb Q}  )$ carries the regular representation of $S_n.$
\end{thm}

We will use the cohomology of $ C\!on\! f_n\, {\mathbb R}^{2} $ as the prototype for the configuration spaces of even-dimensional Euclidean space, and  $ C\!on\! f_n\, {\mathbb R}^{3} $ as the prototype for 
the configuration spaces of odd-dimensional Euclidean space.  Note that cohomology is concentrated in all degrees in the former (more generally in all multiples of $(2d-1)$ for $2d$-dimensional space), and only in even degrees in the latter.

The results of this section will show that the representation $Lie_n^{(2)}$ has properties  curiously parallelling those of $Lie_n.$ 
Theorem \ref{SWThm4.4} above states the ``Lie" identities of
Theorem \ref{Compare} below in the context of the configuration spaces of $X={\mathbb R}^2$ and $X={\mathbb R}^3.$   The module $Lie_n$ arises as the highest nonvanishing cohomology for the configuration space of ${\mathbb R}^d, d$ odd, and when tensored with the sign, as the highest nonvanishing cohomology for the configuration space of ${\mathbb R}^d, d$ even.  This is the classically known 
prototype;  the variant $Lie_n^{(2)}$ will be shown to closely follow its example.  The higher Lie module $H_\lambda[Lie]$ dates back to \cite{T}, and has been studied by several authors in the recent literature.
Note the appearance of the \lq\lq higher $Lie^{(2)}$ modules" below. See  also Theorem~\ref{EquivPBW} for the list of $Lie$ identities.
\begin{thm}\label{Compare}  The  symmetric function $Lie_n^{(2)}$ satisfies the following plethystic identities, analogous to $Lie_n$.

\begin{equation}\label{SymExt}
\sum_{\lambda\vdash n}H_\lambda[Lie]=p_1^n; \qquad\qquad \sum_{\lambda\vdash n}E_\lambda[Lie^{(2)}]=p_1^n;
\end{equation}
\begin{equation}\label{PlInvHE}
H[\sum_{n\geq 1} (-1)^{n-1}\omega(Lie_n)]=1+p_1;
\qquad E[\sum_{n\geq 1} (-1)^{n-1}\omega(Lie^{(2)}_n)]=1+p_1;
\end{equation}

\begin{equation}\label{AcycEH}
\text{If } n\geq 2, \ \sum_{\lambda\vdash n}(-1)^{n-\ell(\lambda)}E_\lambda[Lie]=0; \qquad\qquad \sum_{\lambda\vdash n}(-1)^{n-\ell(\lambda)}H_\lambda[Lie^{(2)}]=0;
\end{equation}

\begin{equation}\label{TotalCoh}
\text{If } n\geq 2, \ \sum_{\lambda\vdash n}E_\lambda[Lie]=2e_2 p_1^{n-2}; \qquad\qquad \sum_{\lambda\vdash n}H_\lambda[Lie^{(2)}]=\sum_{\lambda\vdash n, \lambda_i=2^{a_i}} p_\lambda.
\end{equation}
 Moreover,  the $Lie$ identities are all equivalent, and the $Lie^{(2)}$ identities are also equivalent.
\end{thm}

We now discuss the implications of Theorem \ref{Compare}. 

\vspace{.07in }
\noindent {\bf \textit{Equation (\ref{SymExt})}:}

  The first equation in (\ref{SymExt}) is simply Thrall's classical theorem \cite{T}, rederived in Theorem \ref{ThrallPBWCadoganSolomon}, stating that the regular representation of $S_n$ decomposes into a sum of symmetrised modules induced from the centralisers of $S_n,$ the Lie modules.  Thrall's theorem in this context is equivalent to the Poincar\'e-Birkhoff-Witt Theorem, which states that the universal enveloping algebra of the free Lie algebra is its symmetric algebra \cite{R}.  Recall that the Lefschetz module of a complex is the alternating sum by degree of the homology modules. In view of Theorem \ref{SWThm4.4}, since cohomology is nonzero only in even degrees, the Lefschetz module is in fact a sum of homology modules, and this can in turn be reinterpreted as saying that:
  \begin{prop}\label{LefschetzR3}
   The regular representation of $S_n$  is carried by the Lefschetz module of  $C\!on\!f_n\, {\mathbb R}^3,$ and more generally $C\!on\!f_n\, {\mathbb R}^d$ for odd $d$, which coincides with its cohomology ring and  is  isomorphic to the symmetric algebra over the top cohomology.
   \end{prop}

  The second equation in (\ref{SymExt}) is our new result.  It gives a new decomposition of the regular representation: 
  
  \begin{thm}\label{ExtReg}   The regular representation decomposes into  a sum of \textit{exterior} powers of modules induced from the centralisers of $S_n,$ namely the modules $Lie_n^{(2)}$.
  \end{thm}

\noindent {\bf \textit{Equation (\ref{PlInvHE})}:}

In (\ref{PlInvHE}), the second equation is new, giving the plethystic inverse of the elementary symmetric functions $\sum_{n\geq 1} e_n,$  while 
the first equation contains the known result of Cadogan \cite{C} (see Theorem \ref{ThrallPBWCadoganSolomon}) giving the plethystic inverse of the homogeneous symmetric functions $\sum_{n\geq 1} h_n$.

\vspace{.07in}
\noindent {\bf \textit{Equation (\ref{AcycEH})}:}

The equations in  (\ref{AcycEH}) and (\ref{TotalCoh}) are particularly significant.  It is well known that the degree $n$ term in the plethysm $e_{n-r}[Lie]$ is the Frobenius characteristic of the $r$th-Whitney homology $W\!H_r(\Pi_n)$ of the partition lattice $\Pi_n,$ tensored with the sign (see \cite[Remark 1.8.1]{Su0}),  and hence of the sign-tensored $r$th cohomology $H^r( C\!on\! f_n\,{\mathbb R}^{2} )$ of Theorem \ref{SWThm4.4}. The $r$th Whitney homology also coincides as an $S_n$-module with the $r$th cohomology of the pure braid group, see \cite{HL}.
The first equation in \ref{AcycEH} therefore restates the acyclicity of Whitney homology for the partition lattice \cite{Su0}, and hence also says  (in contrast to the odd case $ C\!on\! f_n\,{\mathbb R}^{3} $ of Proposition~\ref{LefschetzR3} above) that :
\begin{prop}\label{LefschetzR2}
The Lefschetz module of $C\!on\! f_n\,{\mathbb R}^{2}$ (and more generally $C\!on\! f_n\,{\mathbb R}^{2d}$ for even $d$) vanishes identically.
\end{prop}
Writing $W\!H_{odd}(\Pi_n)$ for $\oplus_{k=0}^{n/2} W\!H_{2k+1}(\Pi_n),$ and $W\!H_{even}(\Pi_n)$ for $\oplus_{k=0}^{n/2} W\!H_{2k}(\Pi_n),$ we have the isomorphism of $S_n$-modules
\begin{equation} \label{EvenOdd}W\!H_{odd}(\Pi_n)\simeq W\!H_{even}(\Pi_n), \quad n\geq 2.\end{equation}

\noindent {\bf \textit{Equation (\ref{TotalCoh})}:}

  Denote by $W\!H(\Pi_n)$ the sum of all the graded pieces of the Whitney homology of $\Pi_n.$ 

The first equation in (\ref{TotalCoh}) says (recall that we have tensored with the sign representation) that 
 $W\!H(\Pi_n)=2\, ({\mathbf 1}\uparrow_{S_2}^{S_n}),$  $n\geq 2,$ 
 a result originally due to Lehrer, who proved that this is the $S_n$-representation on the cohomology ring $H^*( C\!on\! f_n\,{\mathbb R}^{2})$  (Lehrer actually considers the cohomology of the complement of the braid arrangement of type $A_{n-1}$ \cite[Proposition 5.6 (i)]{Le}).  We may rewrite this in our notation as 
\begin{equation}\label{Lehrer-a}   H^*( C\!on\! f_n\,{\mathbb R}^{2} )=W\!H(\Pi_n)={\rm ch}^{-1}(2 h_2 p_1^{n-2})=2\ ({\mathbf 1}\uparrow_{S_2}^{S_n}) , \quad n\geq 2.\end{equation}

Note that the first equation in (\ref{TotalCoh}) also confirms the following theorem of Orlik and Solomon.
\begin{prop}\label{Orlik-Solomon} \cite{LS}
$H^*( C\!on\! f_n\,{\mathbb R}^{2}) $ 
has the structure of an exterior algebra over the top cohomology.
\end{prop}

By combining equation (\ref{Lehrer-a}) with (\ref{EvenOdd}), we obtain
\begin{equation}\label{Lehrer-b} H^{odd}(C\!on\! f_n\,{\mathbb R}^{2} )
\simeq H^{even}(C\!on\! f_n\,{\mathbb R}^{2})\simeq {\mathbf 1}\uparrow_{S_2}^{S_n}, \quad n\geq 2,
\end{equation}
yielding the decomposition of the regular representation noticed by   Hyde and Lagarias \cite{HL}:
\begin{equation}\label{HL} 
H^{odd}( C\!on\! f_n\,{\mathbb R}^{2})\oplus 
{\, \bf sgn}_{S_n}\otimes H^{even}(C\!on\! f_n\,{\mathbb R}^{2}) \simeq {\mathbf 1}\uparrow_{S_1}^{S_n}. \end{equation}

From \ref{Lehrer-a} it also follows that 
\begin{equation}\label{IndConf}
H^*( C\!on\! f_{n+1}\,{\mathbb R}^{2})\simeq H^*(C\!on\! f_n\, {\mathbb R}^{2}) \uparrow_{S_{n}}^{S_{n+1}}.
\end{equation}

We now describe results of a similar flavour for the new  representation $Lie_n^{(2)}.$ 
Define a new module $V\!h_r(n)$ whose Frobenius characteristic is the degree $n$ term in $h_{n-r}[Lie^{(2)}];$ this is a true $S_n$-module. The second equation of (\ref{AcycEH}) can now be interpreted as an acylicity statement:
$$V\!h_n(n)-V\!h_{n-1}(n) + V\!h_{n-2}(n)-\ldots +(-1)^r V\!h_r(n) +\ldots=0, \quad n\geq 2.$$
and hence, in analogy with (\ref{EvenOdd}), letting $V\!h_{odd}(n)=\oplus_{k=0}^{n/2} V\!h_{2k+1}$ and $V\!h_{even}=\oplus_{k=0}^{n/2} V\!h_{2k}:$
\begin{equation}\label{EvenOddLieSup2} V\!h_{odd}(n)\simeq V\!h_{even}(n), \quad n\geq 2.\end{equation}

The second equation in (\ref{TotalCoh}) gives, similarly, 
\begin{equation} {\rm ch\,}(V\!h_{odd}(n)\oplus V\!h_{even}(n))=
\sum_{\lambda\vdash n; \lambda_i=2^{a_i}} p_\lambda.
\end{equation}
Hence we have established the following results, analogous to (\ref{Lehrer-a})-(\ref{IndConf}):

\begin{thm}\label{LehrerHLLieSup2}  The following $S_n$-equivariant isomorphisms hold for the modules $Vh_r(n)={\rm ch}^{-1}\, h_{n-r}[Lie^{(2)}]\vert_{{\rm deg\ }n}$, giving   Schur-positive  functions with integer coefficients.
\begin{equation}\label{LehrerSup2b}V\!h_{odd}(n)\simeq V\!h_{even}(n)={\rm ch}^{-1}\,
\frac{1}{2}\sum_{\lambda\vdash n; \lambda_i=2^{a_i}} p_\lambda.
\end{equation}
\vskip-.2in
\begin{equation}\label{LehrerSup2a} V\!h(n)=V\!h_{odd}(n)\oplus V\!h_{even}(n)={\rm ch}^{-1}\,
\sum_{\lambda\vdash n; \lambda_i=2^{a_i}} p_\lambda.
\end{equation}
\vskip-.2in
\begin{equation}\label{HLSup2}
V\!h_{odd}(n)\oplus {\bf sgn}_{S_n}\otimes V\!h_{odd}(n)={\rm ch}^{-1}\,
\sum_{\lambda\vdash n; n-\ell(\lambda) {\text even };\lambda_i=2^{a_i}} p_\lambda.
\end{equation}
\vskip-.2in
\begin{equation}\label{IndSup2}
V\!h(2n+1)\simeq V\!h(2n)\uparrow_{S_{2n}}^{S_{2n+1}}.
\end{equation}
\end{thm}
We now have at least four decompositions of the regular representation, namely the two in (\ref{SymExt}) and two from (\ref{HL}) (tensoring the latter with the sign representation gives two),  into sums of modules indexed by the conjugacy classes,  each module  obtained by  inducing a linear character from a centraliser of $S_n.$  We write these out  for $S_4$ and $S_5$ to show that they are indeed all distinct.    In the two tables below, each column adds up to the regular representation. Note that  $Lie_4^{(2)}$ coincides with $C\!onj_4,$ while $Lie_5^{(2)}$ is just $Lie_5.$ Hence these modules appear in the last row of each table.
The first two decompositions are from equation (\ref{SymExt})  of Theorem \ref{Compare}; the third is from equation (\ref{HL}).
In all cases, of course,  the four pieces for $S_4$ (respectively, the five pieces for $S_5$)  each have the same dimension, equal to the sum of the sizes of the constituent conjugacy classes, namely,  $1,6,11,6$ (respectively $1,10, 35, 50, 24$).  Note that the conjugacy classes are grouped together by number of disjoint cycles, i.e. by length $\ell$ of the corresponding partition.  That these four decompositions are all distinct is clear, since each has a distinguishing feature.  E.g. for  $S_4,$ both copies of the irreducible for the partition $(2^2)$ appear only in one graded piece for {\bf [PBW]}, while the reflection representation is a submodule of one graded piece only in the third.  
\vfill\eject
\begin{center} 
{\small \bf Table 1: The regular representation of $S_4$}\\  \nopagebreak
{\small (Poincar\'e polynomial $1+6t+11t^2+6t^3$)}\end{center}
\begin{center}
\begin{tabular}{|c|c|c|c|}\hline
{\small Conjugacy}  &  {\small\bf PBW} ($C\!on\!f\,{\mathbb R}^3$)
    &  {\small\bf Ext} & {\small\bf Whitney} ($C\!on\!f\,{\mathbb R}^2$) \\
    {\small classes} &{\small irreducibles} &{\small irreducibles} &{\small irreducibles } \\ \hline
    ${\scriptstyle (1^4)}$ & ${\scriptstyle h_4[Lie]|_{\text{deg }4}}$ 
    & ${\scriptstyle e_4[Lie^{(2)}]|_{\text{deg }4}}$
     & ${\scriptstyle\omega(W\!H_0)}$\\
 ${\scriptstyle \ell=4}$   &${\scriptstyle (4)}$ &${\scriptstyle (1^4)}$ &${\scriptstyle (1^4)}$\\ \hline
    ${\scriptstyle (2,1^2)}$ & ${\scriptstyle h_3[Lie]|_{\text{deg }4}}$ & ${\scriptstyle e_3[Lie^{(2)}]|_{\text{deg }4}}$ 
    & ${\scriptstyle W\!H_1}$\\
  ${\scriptstyle \ell=3}$  &${\scriptstyle (3,1)+(2,1^2)}$ & ${\scriptstyle (3,1)+(2,1^2)}$ & ${\scriptstyle (4)+(3,1)+(2^2)}$\\ \hline
    ${\scriptstyle (3,1)\text{ and }(2^2)}$ &${\scriptstyle h_2[Lie]|_{\text{deg }4}}$
    &${\scriptstyle e_2[Lie^{(2)}]|_{\text{deg }4}}$
    &${\scriptstyle\omega(W\!H_2)}$\\
 ${\scriptstyle \ell=2}$   & ${\scriptstyle (3,1)+2(2^2)+(2,1^2) +(1^4)}$
    &${\scriptstyle 2(3,1)+(2^2)+(2,1^2)}$ 
    &${\scriptstyle (3,1)+2 (2,1^2)+(2^2)}$\\ \hline
    ${\scriptstyle (4)}$ & ${\scriptstyle h_1[Lie]|_{\text{deg }4}}$ 
    &${\scriptstyle e_1[Lie^{(2)}]|_{\text{deg }4}}$
    &${\scriptstyle W\!H_3=\omega(Lie_4)}$ \\ 
 ${\scriptstyle \ell=1}$   & ${\scriptstyle (3,1)+(2,1^2)}$ &${\scriptstyle (4)+(2^2)+(2,1^2)} $ &${\scriptstyle (3,1)+(2,1^2)}$\\ \hline
\end{tabular}
\end{center}
\begin{center} {\small \bf Table 2: The regular representation of $S_5$ }\\ {\small (Poincar\'e polynomial $1+10t+35t^2+50t^3+24t^4 $)}\end{center}
\nopagebreak
\begin{center}
\begin{tabular}{|c|c|c|c|}\hline
{\small Conjugacy}  &  {\small\bf PBW}($C\!on\!f\,{\mathbb R}^3$)
    &  {\small\bf Ext} & {\small\bf Whitney}($C\!on\!f\,{\mathbb R}^2$) \\
   {\small  classes} &{\small irreducibles} &{\small irreducibles} &{\small irreducibles } \\ \hline
    ${\scriptstyle (1^5)}$ & ${\scriptstyle h_5[Lie]|_{\text{deg }5}}$ 
    & ${\scriptstyle e_5[Lie^{(2)}]|_{\text{deg }5}}$
     & ${\scriptstyle\omega(W\!H_0)}$\\
 ${\scriptstyle \ell=5}$   &${\scriptstyle (5)}$ &${\scriptstyle (1^5)}$ &${\scriptstyle (1^5)}$\\ \hline
    ${\scriptstyle (2,1^3)}$ & ${\scriptstyle h_4[Lie]|_{\text{deg }5}}$ & ${\scriptstyle e_4[Lie^{(2)}]|_{\text{deg }5}}$ 
    & ${\scriptstyle W\!H_1}$\\
${\scriptstyle \ell=4}$    &${\scriptstyle (4,1)+(3,1^2)}$ & ${\scriptstyle (3,1^2)+(2,1^3)}$ & ${\scriptstyle (5)+(4,1)+(3,2)}$\\ \hline
    ${\scriptstyle (3,1^2) \text{ and }(2^2,1)}$ &${\scriptstyle h_3[Lie]|_{\text{deg }5}}$
    &${\scriptstyle e_3[Lie^{(2)}]|_{\text{deg }5}}$
    &${\scriptstyle\omega(W\!H_2)}$\\
 ${\scriptstyle \ell=3}$   & ${\scriptstyle (4,1)+2(3,2)+(3,1^2)}$ 
    &${\scriptstyle (4,1)+2(3,2)+2(3,1^2)}$
    &${\scriptstyle (3,2)+2(3,1^2)}$\\ 
&${\scriptstyle +2(2^2,1)+(2,1^3)+(1^5)}$ &${\scriptstyle +(2^2,1)+(2,1^3)}$ & ${\scriptstyle +2(2^2,1)+2(2,1^3)
}$\\
\hline
      ${\scriptstyle (4,1)\text{ and }(3,2)}$ &${\scriptstyle h_2[Lie]|_{\text{deg }5}}$  & ${\scriptstyle e_2[Lie^{(2)}]|_{\text{deg }5}}$
      &${\scriptstyle W\!H_3}$ \\ 
 ${\scriptstyle \ell=2}$     &${\scriptstyle (4,1)+2(3,2)+3(3,1^2)}$ &${\scriptstyle (5)+2(4,1)+2(3,2)}$ &${\scriptstyle 2(4,1)+2(3,2)+3(3,1^2)}$\\ 
      &${\scriptstyle +2(2^2,1)+2(2,1^3)}$ & ${\scriptstyle +2(3,1^2)+3(2^2,1)+(2,1^3)}$ 
      &${\scriptstyle +2(2^2,1)+(2,1^3)}$\\
      \hline
    ${\scriptstyle (5)}$ & ${\scriptstyle h_1[Lie]|_{\text{deg }5}}$ 
    &${\scriptstyle e_1[Lie^{(2)}]|_{\text{deg }5}}$
    &${\scriptstyle  \omega W\!H_4=Lie_5}$ \\ 
 ${\scriptstyle \ell=1}$   & ${\scriptstyle (4,1)+(3,2)+(3,1^2)}$ 
    &${\scriptstyle (4,1)+(3,2)+(3,1^2)}$ 
    & ${\scriptstyle (4,1)+(3,2)+(3,1^2)}$\\ 
    &${\scriptstyle +(2^2,1)+(2,1^3)}$ &${\scriptstyle +(2^2,1)+(2,1^3)}$ &${\scriptstyle +(2^2,1)+(2,1^3)}$ \\
    \hline
\end{tabular}
\end{center}
Note from the above example that the two identities in equation (\ref{SymExt}) of Theorem \ref{Compare}, corresponding respectively to (\ref{PBW}) and (\ref{Ext}) below,  themselves yield the following four distinct decompositions of the regular representation, obtained by tensoring each graded piece with the sign representation.  
The decomposition in equation (\ref{Eulerian idempotents}) below is precisely that obtained from the Eulerian idempotents of Gerstenhaber and 
Schack   \cite{GS}; this fact was proved by Hanlon \cite[Theorem 5.1 and Definition 3.6]{Ha}.  Curiously it also appears in a paper of Gessel, Restivo and Reutenauer \cite[Lemma 5.3, Theorem 5.1]{GRR}, where the authors give a combinatorial decomposition of the full tensor algebra as the enveloping algebra of the \textit{oddly generated free Lie superalgebra}; they call equation (\ref{Eulerian idempotents}) below a \lq\lq super\rq\rq version of the Poincar\'e-Birkhoff-Witt theorem.

We have, for $n\geq 1:$
\begin{align}
p_1^n&= \sum_{k\geq 1}\sum_{\stackrel{\lambda\vdash n}{\ell(\lambda)=k}} H_\lambda[Lie]\qquad\qquad \text{(PBW)}\label{PBW}\\
&= \sum_{k\geq 1}\sum_{\stackrel{\lambda\vdash n}{\ell(\lambda)=k}} \omega(H_\lambda[Lie])\qquad \text{(Eulerian idempotents)}\label{Eulerian idempotents}\\
&= \sum_{k\geq 1}\sum_{\stackrel{\lambda\vdash n}{\ell(\lambda)=k}} E_\lambda[Lie^{(2)}]\qquad\quad \text{(Ext)}\label{Ext}\\
&= \sum_{k\geq 1}\sum_{\stackrel{\lambda\vdash n}{\ell(\lambda)=k}} \omega(E_\lambda[Lie^{(2)}])
\end{align}
Example 5.3 shows that these four decompositions are themselves distinct, and also distinct from the two decompositions arising from the Whitney homology of the partition lattice.

We point out one more analogy between $W\!H_k(\Pi_n)\simeq H^k(C\!on\!f\, {\mathbb R}^2)$ and the modules $V_k(n)$  arising from the identities of Theorem \ref{Compare}.  In \cite{Su0}, it was shown that the Whitney homology of the partition lattice (and more generally of any  Cohen-Macaulay poset) has the following important property:
\begin{thm}\label{betas}\cite[Proposition 1.9]{Su0} For $0\leq k\leq n-1,$ the truncated alternating sum 
$$W\!H_k(\Pi_n)-W\!H_{k-1}(\Pi_n)+\ldots+(-1)^k W\!H_0(\Pi_n)$$ is a true $S_n$-module, and is isomorphic as an $S_n$-module to the unique nonvanishing homology of the rank-selected subposet of $\Pi_n$ obtained by selecting the first $k$ ranks. Equivalently, 
the degree $n$ term in the plethysm 
$$(e_{n-k}-e_{n-k+1}+\ldots +(-1)^{k} e_n)[Lie]$$ is Schur-positive. In particular, 
the $k$th Whitney homology decomposes into a sum of two $S_n$-modules as follows:
$${\rm ch\,} W\!H_k(\Pi_n)=\omega\left(e_{n-k}[Lie]|_{\text{deg }n}\right)=\beta_n([1,k])+\beta_n([1,k-1]),$$ 
where  $\beta_n([1,k])$ denotes  the Frobenius characteristic of the rank-selected homology of the first $k$ ranks of $\Pi_n$ as in \cite[Proposition 1.9]{Su0}.
\end{thm}

We conjecture that a similar decomposition exists for the $S_n$-modules $V\!h_k(n).$ More precisely, we have 

\begin{conj}\label{LieSup2betas}  Let $V\!h_k(n)$ be the $S_n$-module whose Frobenius characteristic is the degree $n$ term in the plethysm $h_{n-k}[Lie^{(2)}],$ for $k=0,1,\ldots, n-1.$  Then for $0\leq k\leq n-1,$ the truncated alternating sum 
$$V\!h_k(n)-V\!h_{k-1}(n)+\ldots+(-1)^k V\!h_0(n)=U_k(n)$$ is a true $S_n$-module, and hence one has the $S_n$-module decomposition
$$ V\!h_k(n)={\rm ch}^{-1}\,h_{n-k}[Lie^{(2)}]|_{\text{deg }n}\simeq  U_k(n) + U_{k-1}(n).$$
(Here we define $U_{-1}(n)$ to be the zero module and $U_0(n)$ to be the trivial $S_n$-module.
Equivalently, the degree $n$ term in the  plethysm
$$(h_{n-k}-h_{n-k+1}+\ldots +(-1)^{k} h_n)[Lie^{(2)}]$$ is Schur-positive  for $0\leq k \leq n-1.$
\end{conj}
 
This conjecture is easily verified for $0\leq k\leq 3;$ in the latter case there are relatively simple formulas for ${\rm ch\,} V\!h_k(n),$ giving the following for $U_k(n),$ (for $n\geq 4$). 
\begin{align*}{\rm ch\,}U_0(n)&={\rm ch\,}V\!h_0(n)=h_n;\\
 {\rm ch\,}U_1(n)&=(h_{n-1}-h_n)[Lie^{(2)}]|_n= h_2h_{n-2}-h_n= s_{(n-1,1)}+s_{(n-2,2)};\\
{\rm ch\,}U_2(n)&={\rm ch\,}V\!h_2(n)-{\rm ch\,}U_1(n) 
=h_{n-2}s_{(2,1)}-s_{(n-1,1)}-s_{(n-2,2)} +h_{n-4}(h_4+s_{(2,2)}),\\
&\text{which is clearly Schur-positive by the Pieri rule;}\\
{\rm ch\,}U_3(n)&={\rm ch\,}V\!h_3(n)-{\rm ch\,}U_2(n)\\
&=h_{n-4} s_{(2,1^2)} +s_{(2,1)} (h_{n-5} h_2-h_{n-3}) +h_{n-6}(h_6+s_{(4,2)}+s_{(2^3)})\\
&+s_{(n-1,1)}+s_{(n-2,2)},
\text{ which is again clearly Schur-positive for }n\geq 5.
\end{align*}

We include the data for $n=6$ and $n=7$ below.
\vskip .2in
\begin{center}{\small \bf Table 3: Alternating sums $U_k(n)$ of $h_k[Lie^{(2)}]$ for $n=6$}\end{center}
\begin{tabular}{|c|c|}\hline
$k$ & $U_k(6)$\\[3.5pt] \hline
0& {$(6)$}\\[3pt] \hline
1  &${(5,1)}+{(4,2)}$\\[3.5pt] \hline
2  &${(6)}+{(5,1)}+2{(4,2)}
 +{(4,1^2)} +2{(3,2,1)}+{(2^3)}$\\[3.5pt] \hline
3 & ${(6)}+{(5,1)}+3{(4,2)}+2{(4,1^2)}
+{(3^2)}+3{(3,2,1)}+2{(3,1^3)}+2 {(2^2,1^2)}$\\[3.5pt] \hline
4 & $Lie_6^{(2)}={(5,1)}+2{(4,2)}+{(4,1^2)}+3{(3,2,1)}+2{(3,1^3)}+{(2^3)}+ {(2^2,1^2)}+{(2,1^4)}$\\[3.5pt] \hline
\end{tabular}

\begin{center}{\small \bf Table 4: Alternating sums $U_k(n)$ of $h_k[Lie^{(2)}]$ for $n=7$}\end{center}
\nopagebreak
\begin{tabular}{|c|c|} \hline
$k$ & $U_k(7)$\\ \hline 
0 & ${(7)}$\\ \hline
1 & ${(6,1)}+{(5,2)}$\\ \hline
 2 & $(7)+(6,1)+2(5,2)+(5,1^2) +(4,3)+2(4,2,1)+(3,2^2)$\\ \hline
3 & $(7) +2 (6, 1) +3 (5, 2) +2 (5, 1^2) + 3 (4, 3) +5 (4, 2, 1) 
+2 (4, 1^3)+  2 (3^2, 1)+3 (3, 2^2) $\\
&$ + 3 (3, 2, 1^2)
+ 2 (2^3, 1)$\\  \hline
4 & $2 (6, 1)+ 4 (5, 2)+3 (5, 1^2)+3 (4,3) +8 (4,2,1) +3 (4,1^3)   +  4 (3^2, 1)+  5 (3, 2^2)$\\
&$ + 7 (3, 2, 1^2)
 + 3 (3, 1^4)+ 3 (2^3, 1) +    2 (2^2, 1^3)$\\ \hline
 5 & $ Lie^{(2)}_7=Lie_7=(6, 1)+2 (5, 2)+2 (5, 1^2) + 2 (4, 3) +5 (4, 2, 1) +3 (4, 1^3)$\\
 &$+3 (3^2, 1)+3 (3, 2^2) + 5 (3, 2, 1^2) +2 (3, 1^4)+2 (2^2, 1)+ 2 (2^2, 1^3)   +(2, 1^5) $\\ \hline
\end{tabular}

\vskip .2in

  Recent work of Hyde and Lagarias \cite{HL} rediscovers the representations $\beta_n([1,k])$ of Theorem \ref{betas} in a cohomological setting.
Our results  suggest the existence of a similar topological context in which the modules $V\!h_k(n)$ and $U_k(n)$ appear.

\begin{qn} Is there a cohomological context for the ``$Lie^{(2)}$" identities of Theorem \ref{SymExt}, as there is for the $Lie$ identities in the context of configuration spaces (Theorem \ref{SWThm4.4}), or as in  \cite{HL}?
\end{qn}

Recall from Section 2 and Theorem \ref{EquivPBW} the following facts.  The free Lie algebra has a filtration arising from its derived series \cite[Section 8.6.12]{R}, which in our notation may be described as follows.
Let $\kappa=\sum_{n\geq 2} s_{(n-1,1)}.$  Then     
$Lie_{\geq 2}=\kappa+\kappa[\kappa]+ \kappa[\kappa[\kappa]]+\ldots.$ 

Theorem \ref{SymExt} allows us to  deduce a similar   decomposition for $Lie_n^{(2)}.$  In fact we have the following exact analogue of Theorem \ref{EquivPBW}:

\begin{thm}\label{EquivLieSup2}  The following identities hold, and are equivalent:
\begin{equation}\label{ExtLieSup2}  (E-1)[Lie^{(2)}]=(\sum_{r\geq 1} e_r)[Lie^{(2)}]=\sum_{n\geq 1} p_1^n. \end{equation}
\begin{equation}\label{PlInvLieSup2} (1-H^{\pm})[Lie^{(2)}]=(\sum_{r\geq 1} (-1)^{r-1} h_r)[Lie^{(2)}]=p_1. \end{equation}
\begin{equation}\label{Extge2} (1-H^{\pm})[Lie_{\geq 2}^{(2)}]=(\sum_{r\geq 1} (-1)^{r-1} h_r)[Lie_{\geq 2}^{(2)}]=\omega(\kappa) \end{equation}
\begin{equation}\label{LieSup2cochain}  \text{ The degree $n$ term in }\sum_{r\geq 0} (-1)^{n-r} h_{n-r}[Lie_{\geq 2}^{(2)}] \text{ is } (-1)^{n-1} s_{(2,1^{n-2})}.
\end{equation}
\begin{equation}\label{LieSup2Filt_a}   Lie_{\geq 2}^{(2)}=Lie^{(2)}[\omega(\kappa)]\end{equation}
\begin{equation} \label{LieSup2Filt_b}
 Lie_{\geq 2}^{(2)}=\omega(\kappa)+\omega(\kappa)[\omega(\kappa)]+\omega(\kappa)[\omega(\kappa)[\omega(\kappa)]]+\ldots 
  \end{equation}
   (Analogue of the derived series filtration of the free Lie algebra)  
 \begin{equation}\label{HodgeLieSup2} (E-1)[Lie^{(2)}_{\geq 2}]=\sum_{r\geq 1} e_r[Lie^{(2)}_{\geq 2}]=(1-p_1)^{-1}\cdot H^{\pm}-1
 =\sum_{n\geq 2}\sum_{k=0}^n (-1)^k p_1^{n-k}h_k. \end{equation}
\end{thm}

We offer two more contrasting results for $Lie_n$ and $Lie_n^{(2)}:$
\begin{prop}\label{AltHLieAltELieSup2}  Let $D\!Par$ denote the set of partitions with distinct parts.  

\begin{enumerate}[itemsep=8pt]
\item $\sum_{r\geq1} (-1)^{r-1} h_r[Lie]|_{{\rm deg\ }n}=
\begin{cases} p_1 p_2^k, & n=2k+1 \text{ is odd}\\
                    -p_2^k, & n=2k \text{ is even}
                    \end{cases}$
\item     $\sum_{r\geq1} (-1)^{r-1} e_r[Lie^{(2)}]|_{{\rm deg\ }n}=               
\sum_{\stackrel{\lambda\vdash n:\lambda_i=2^{k_i}, k_i\geq 0}{\lambda\in D\!Par}}
(-1)^{\ell(\lambda)-1} p_\lambda.$

\end{enumerate}
\end{prop}

Next we examine more closely the action on  derangements, i.e. fixed-point-free permutations.
  Reiner and Webb study the Cohen-Macaulay complex of injective words, and compute the $S_n$-action on its top homology \cite{RW}. Theorem \ref{EquivLieSup2} shows that the representations $Lie_n^{(2)}$ make an appearance here as well:

\begin{thm}\label{Lie2Hodge}
Let $n\geq 2.$ For $k\geq 1$ let 
$\Delta_n^k$ denote the degree $n$ term in $e_k[Lie_{\geq 2}^{(2)}].$ Define $\Delta_n=\sum_{k\geq 1} \Delta_n^k  \text{ for } n\geq 2,$ and $ \Delta_1=0, \Delta_0=1.$ 
Then  \begin{enumerate} 
\item  $\Delta_n=\sum_{k=0}^n (-1)^k p_1^{n-k} h_k=p_1\Delta_{n-1}+(-1)^n h_n;$ and hence 
\item For $n\geq 2,$ $\Delta_n$ coincides with the Frobenius characteristic of the homology representation on the complex of injective words in the alphabet $\{1,2,\ldots,n\}.$ 
\end{enumerate}
\end{thm}
\begin{proof} Clearly  $\Delta_n$ is the degree $n$ term in $E[Lie_{\geq 2}^{(2)}],$ so this is nothing but a restatement of equation (\ref{HodgeLieSup2}) above.
\end{proof}

Hanlon and Hersh showed that this homology representation has a Hodge decomposition \cite[Theorem 2.3]{HH}, by showing that the complex itself splits into a direct sum of $S_n$-invariant subcomplexes. Writing $D_n^k$ for the degree $n$ term in $h_k[Lie_{\geq 2}],$ in our terminology their result may be stated as follows:
$$\Delta_n=\sum_{k\geq 1} \omega(D_n^k).$$
In fact  the  identity $\sum_k D_n^k=\sum_{k=0}^n (-1)^k p_1^{n-k} e_k$ is simply a restatement of equation \eqref{LieHodge} in Theorem \ref{EquivPBW}.

Surprisingly, the decomposition of $\Delta_n$ given in Theorem \ref{Lie2Hodge} is different from the Hodge decomposition, i.e. the summands $\Delta_n^k$ and $\omega(D_n^k)$ do not coincide.  The first nontrivial example appears below. 
\begin{ex}  For $n=4,$ we have $\Delta_4=p_1^2h_2-p_1h_3+h_4
=(4)+(3,1)+(2^2)+(2,1^2).$ Also $\Delta_4^2=e_2[h_2]=(3,1),$ 
$\Delta_4^1=Lie_4^{(2)}=(4)+(2^2)+(2,1^2).$ 
The two Hodge pieces, however, each consist of two irreducibles:
$\omega(h_2[Lie_2])=(2^2)+(4)$ and $\omega(h_1[Lie_4])=(3,1)+(2,1^2).$
\end{ex}

This prompts the following:
\begin{qn}  Is there an algebraic complex  explaining the representation-theoretic decomposition 
$$ \Delta_n=\sum_{k\geq 0}\Delta_n^k=\sum_{k\geq 0} e_k[Lie_{\geq 2}^{(2)}]|_{{\rm deg\ } n},$$
just as the Hodge complex explains the decomposition
$$\Delta_n =\omega(\sum_{k\geq 0} h_k[Lie_{\geq 2}]|_{{\rm deg\ } n}),$$
noting (from the preceding example) that $e_k[Lie_{\geq 2}^{(2)}]|_{{\rm deg\ } n}$ is not in general equal to 
$\omega( h_k[Lie_{\geq 2}]|_{{\rm deg\ } n})$?
\end{qn}

It is a well-known fact ( see \cite{LS}, \cite{Su1}) that the exterior power of $Lie,$ when tensored with the sign representation, coincides with the Whitney homology of the lattice of set partitions.
Thus equation (\ref{EquivPBWAltExtge2b}) (from the fundamental theorem of equivalences, Theorem \ref{EquivPBW}), when tensored with the sign, can be rewritten as a formula for the alternating sum of Whitney homology modules of $\Pi_n$, when restricted to partitions with no blocks of size 1.  Define 
$W\!H_{\geq 2}^i(\Pi_n)$ to be the sum of  all the homology modules $\tilde{H}(\hat 0, x)$ where $x$ ranges over all partitions into $n-i$ blocks, with no blocks of size 1. Then 
${\rm ch\ } W\!H_{\geq 2}^i(\Pi_n)=\omega( e_{n-i}[Lie_{\geq 2}]|_{{\rm deg\ } n}) $ and so equation \eqref{EquivPBWAltExtge2b} (and hence the Poincar\'e-Birkhoff-Witt theorem), is equivalent to 
\begin{equation} \sum_{i\geq 0} (-1)^i {\rm ch\,}W\!H_{\geq 2}^i(\Pi_n) = (-1)^{n-1} s_{(2,1^{n-2})}.
\end{equation}
In the notation of \cite{HR}, $\widehat{W}^i_n=W\!H_{\geq 2}^i(\Pi_n)$ (see Corollary 2.11). Hersh and Reiner construct an $S_n$-cochain complex $F_n(A^*)$ with nonvanishing cohomology only in degree $n-1,$ whose $S_n$-character is the irreducible indexed by $(2, 1^{n-2}),$ explicitly proving a conjecture of Wiltshire-Gordon (\cite[Conjecture 1.5, Theorem 1.6, Theorem 1.7]{HR}).

Define, in analogy with \cite{HR}, $\widehat{V_n(k)}$ to be the module with Frobenius characteristic $h_{n-k}[Lie_{\geq 2}^{(2)}]|_{{\rm deg\ }n}.$  Then it is natural to ask:

\begin{qn}  Is there an $S_n$-(co)chain complex for the representations $Lie_{\geq 2}^{(2)}$ whose Lefschetz module  
is given by equation (\ref{LieSup2cochain}) of Theorem \ref{EquivLieSup2} above, i.e. the analogue of equation (2.6)? 
Note that although the nonvanishing (co)homology would occur again only in degree $(n-1)$, affording the same irreducible indexed by $(2, 1^{n-2}),$ the modules in the alternating sum are now different (although they are once again obtained by inducing one-dimensional modules from the same centralisers of $S_n,$ and thus have the same dimensions).  More precisely, and curiously, 
$$\widehat{W}^i_n\not\simeq \widehat{V_n(i)},$$
 although in both cases the alternating sums collapse to the irreducible indexed by $(2, 1^{n-2})$.
 For instance, the calculation for  $n=4$ gives:
 $$\omega(e_{1}[Lie_{\geq 2}]|_{{\rm deg\ }4})=\omega(Lie_4)=Lie_4\neq h_{1}[Lie_{\geq 2}^{(2)}]|_{{\rm deg\ }4}=Lie_4^{(2)};$$
 $$\omega(e_{2}[Lie_{\geq 2}]|_{{\rm deg\ }4})=\omega(e_2[e_2])=e_2[h_2]\neq h_{2}[Lie_{\geq 2}^{(2)}]|_{{\rm deg\ }4}=h_2[h_2] .$$
\end{qn}
We summarise these facts in the following:
\begin{thm}\label{HodgeFilt}  We have
\begin{enumerate}
\item (Hodge decomposition for complex of injective words) 
\begin{equation}\label{HodgeInj}
\sum_{r\geq 1} \omega\left(h_r[Lie_{\geq 2}]|_{{\rm deg\ }n}\right)
=\sum_{k=0}^n (-1)^k p_1^{n-k} h_k 
=  \sum_{r\geq 1} e_r[Lie^{(2)}_{\geq 2}]|_{{\rm deg\ }n}  ;
\end{equation}
\item (Derived series filtration)
\begin{equation}\label{LieSupFiltration}
 \sum_{r\geq 1} (-1)^{r-1}\omega\left(e_r[Lie_{\geq 2}]|_{{\rm deg\ }n}  \right)
=s_{(2,1^{n-2})}
=\sum_{r\geq 1} (-1)^{r-1}h_r[Lie^{(2)}_{\geq 2}]|_{{\rm deg\ }n}.
\end{equation}

\end{enumerate}
\end{thm}

By applying  Part (1) of Theorem \ref{Restrictge2} to  $F=Lie^{(2)},$ we obtain the following analogue of a result of \cite{HR}.  See the remarks at the end of Section~\ref{SecMetaThms}.

\begin{prop}\label{Restrictge2LieSup2} Let $\alpha_n=H[Lie^{(2)}_{\geq 2}]|_{{\rm deg\ }n}, n\geq 0$ We have $\alpha_0=1, \alpha_1=0.$
Then $\alpha_n=p_1\cdot \alpha_{n-1} +(-1)^n \sigma_n,$ 
where $\sigma_n=\sum_{i\geq 0}e_{n-2i}g_{2i}.$ Here  
$g_n$ is the virtual representation of dimension zero given by 
$g_n=\sum_\lambda p_\lambda,$ the sum running over all partitions $\lambda$ of $n$ with no part equal to 1, and all parts a power of 2.  In particular $\sigma_n$ is the characteristic of a one-dimensional virtual representation whose restriction to $S_{n-1}$ is $\sigma_{n-1}.$
\end{prop}

The first few virtual representations  $\sigma_n$ are $\sigma_0=1, \sigma_2=e_2+p_2=s_{(2)},
\sigma_3=e_3+e_1p_2=2s_{(1^3)}+s_{(3)},
\sigma_4=2s_{(4)}-s_{(3,1)}+s_{(2^2)}, \sigma_5=2s_{(5)}-s_{(3,1^2)}+s_{(2^2,1)}.$

The analogous recurrence for the exterior powers $E[Lie^{(2)}_{\geq 2}]|_{{\rm deg\ }n}, n\geq 0$, 
has already been stated in (1) of Theorem \ref{Lie2Hodge}.
See also the remark at the end of Section ~\ref{SecMetaThms}.

We conclude with yet another feature of the $Lie_n$ representation which seems to be shared to some extent by $Lie_n^{(2)}$.
Recall that $Lie_{n-1}\otimes \textbf{ sgn}$ admits a lifting $W_{n}$ which is a true $S_n$-module, the Whitehouse module, appearing in many different contexts \cite{RWh}, \cite{Wh},  whose Frobenius characteristic  is given by 
${\rm ch\,} W_{n}=p_1 \omega(Lie_{n-1}) -\omega(Lie_{n}).$  (See also \cite[Solution to Exercise 7.88 (d)]{St4EC2} for more extensive references.)

One can ask if the same construction for $Lie_n^{(2)}$ yields a true 
$S_{n}$-module.  Clearly one obtains a possibly virtual module which restricts to $Lie_{n-1}^{(2)}$ as an $S_{n-1}$-module. We have the following conjecture, verified in Maple (with Stembridge's  SF package) up to $n=32:$

\begin{conj} The symmetric function $p_1 Lie_{n-1}^{(2)}- Lie_n^{(2)}$ 
is Schur-positive if and only if $n$ is NOT a power of 2.  Equivalently, 
$ Lie_{n-1}^{(2)}\uparrow^{S_n}- Lie_n^{(2)}$ is a true $S_n$-module which lifts $ Lie_{n-1}^{(2)},$ if and only if $n$ is not a power of 2.
\end{conj}

One direction of this conjecture is easy to verify.  Let $n=2^k.$ Then 
$n-1$ is odd, so $Lie_{n-1}^{(2)}=Lie_n.$ Also $Lie_n^{(2)}={\rm ch\,} \textbf{ 1}\uparrow_{C_n}^{S_n}= C\!onj_n,$ i.e. $Lie_n^{(2)}$ is just the permutation module afforded by the conjugacy action on the class of $n$-cycles of $S_n.$  Consequently it contains the trivial representation (exactly once). But it is well known that $Lie_n$ never contains the trivial representation, and hence, when $n$ is a power of 2,  the trivial module appears with negative multiplicity $(-1)$ in $p_1 Lie_{n-1}^{(2)}- Lie_n^{(2)}.$ 

\section{The Frobenius characteristic of $Lie_n^{(2)}$}

In this section we will derive the key symmetric function identities satisfied by the Frobenius characteristic of the module $Lie_n^{(2)},$ thereby proving its intriguing parallelism with $Lie_n.$

 We begin with a general theorem of Foulkes  on the character values of representations induced from the cyclic subgroup $C_n$ of $S_n,$ 
asserting Part (1) of the following (see also \cite[Ex. 7.88]{St4EC2}). We refer the reader to \cite{St4EC2} for the definition of the major index statistic on tableaux.  

\begin{thm}\label{Foulkes}   Let $\ell_n^{(r)}$ denote the Frobenius characteristic of the induced representation $\exp\left(\frac{2i\pi}{n}\cdot r\right)\big\uparrow_{C_n}^{S_n},$  $1\leq r  \leq n.$  Then 
\begin{enumerate}
\item (Foulkes) \cite{F}
$$\ell_n^{(r)}=\dfrac{1}{n}  \sum_{d|n} \phi(d) \dfrac{\mu(\tfrac{d}{(d, r)})} {\phi(\tfrac{d}{(d, r)})} p_d^\frac{n}{d}.$$
\item (Stanley; Kr\'{a}skiewicz and Weyman) (\cite{KW},  \cite{St4EC2}; see also \cite{R}) The multiplicity of the Schur function $s_\lambda$ in the Schur function expansion of $\ell_n^{(r)}$ is the number of standard Young tableaux of shape $\lambda$ with major index congruent to $r$ modulo $n.$
\end{enumerate}
\end{thm}

\noindent
\textit{Remark.}
The quantity $\phi(d) \mu(\tfrac{d}{(d, r)})/ \phi(\tfrac{d}{(d, r)}) $ in Foulkes' formula is called a \textit{Ramanujan sum}; it is the sum of the $r$th powers of the primitive $d$th roots of unity.

Thus $Lie_n$ and $C\!onj_n$ are obtained by taking $r=1$ and $r=n$ in Foulkes' theorem, while our new variant $Lie_n^{(2)}$ is the case $r=k_n,$ where $k_n$ is the highest power of 2 dividing $n.$
Note  that Part (2) provides a complete combinatorial description of the decomposition into irreducibles of $Lie_n, C\!onj_n$ and also  $Lie_n^{(2)}.$ 

Our goal in this section is to describe the symmetric and exterior powers of $Lie_n^{(2)},$  the analogues of the higher Lie modules in Section 2. 
The meta theorem, Theorem \ref{metathm}, of Section~\ref{SecMetaThms} allows us to deduce formulas for these higher $Lie_n^2$-modules quickly and elegantly, avoiding technical plethystic or cycle index calculations.  
We begin by stating three well-known results on $Lie_n$ and $C\!onj_n.$  

\begin{thm} \label{ThrallPBWCadoganSolomon}\cite{T, C, So} (See also \cite[Ex. 7.71, Ex. 7.88, Ex. 7.89]{St4EC2}.) The symmetric powers of $Lie_n$ and $C\!onj_n$ satisfy the following:
\begin{equation*}\label{ThrallPBW}  (Thrall,\ PBW) \qquad H[\sum_{n\geq 1} Lie_n](t)=(1-tp_1)^{-1}
\end{equation*}
(Decomposition of the regular representation into a sum of higher Lie modules)
 \begin{equation*}\label{Cadogan} (Cadogan)\qquad  H[\sum_{n\geq 1} (-1)^{n-1} \omega(Lie_n)](t)
=1+tp_1. \end{equation*} 
(The plethystic inverse of $\sum_{n\ge 1} h_n.$)
\begin{equation*}\label{Solomon} (Solomon)\qquad H[\sum_{n\geq 1}  C\!onj_n](t)=\prod_{n\geq 1} (1-t^np_n)^{-1}
\end{equation*}
\end{thm}

\begin{prop}\label{ExtLieConj} \cite[Theorem 4.2 and Corollary 5.2]{Su1} The exterior powers of $Lie$ and $C\!onj$ satisfy the following:
\begin{enumerate}[itemsep=8pt]
\item $E[Lie](t)=(1-t^2 p_2)(1-tp_1)^{-1}$
\item 
$\sum_{\lambda\in Par} (-1)^{|\lambda|-\ell(\lambda)} H_\lambda[Lie] (t)=\omega(E[\omega(Lie)^{alt}])(t)=(1+tp_1)(1-t^2p_2)^{-1}$

\item $E[\sum_{n\geq 1} C\!onj_n](t)=
\prod_{n\geq 1,\, n\, odd} (1-t^np_n)^{-1}$

\item 
$E[\sum_{n\geq 1} (-1)^{n-1} \omega(C\!onj_n)](t)=
\prod_{n\geq 1,\, n\, odd} (1+t^np_n)$

\end{enumerate}
\end{prop}

Recall that write $Lie_n^{(2)}$ for the Frobenius characteristic of the representation.  Also recall the definition of the polynomial $Lie_n^{(2)}(t)$ from ~\eqref{definepolyf_n}.
\begin{lem}\label{Liesup2atPlusminus1}  Let $n=2^\alpha\ell$ where $\ell$ is odd.  We have
\begin{equation}\label{Liesup2-1}Lie_n^{(2)}=\frac{1}{n}\sum_{s=0}^\alpha \sum_{\stackrel{d=2^s d_1}{d_1|\ell}} \phi(2^s)\mu(d_1) p_{d_1}^{\frac{n}{d_1}},\end{equation}
\begin{equation}\label{Liesup2-2} Lie_n^{(2)}(1)=\begin{cases} 1, & n=2^\alpha \text{ for some }\alpha\geq 0,\\
0, &\text{ otherwise.}\end{cases}
\end{equation}
\begin{equation}\label{Liesup2-3} Lie_n^{(2)}(-1)=\begin{cases} -1, & n=1,\\
0, &\text{ otherwise.}\end{cases}
\end{equation}
\end{lem}

\begin{proof}  Equation~\eqref{Liesup2-1} follows directly from Foulkes' formula, by factoring the highest power of 2 out of each divisor $d$ of $n,$ and using the multiplicativity of $\phi$ and $\mu.$ 
Hence we have 
\[Lie_n^{(2)}(1)=\frac{1}{n} \sum_{d_1|\ell} \mu(d_1) (1+\sum_{s=1}^\alpha \phi(2^s))=\frac{1}{n} \sum_{d_1|\ell} \mu(d_1) (1+\sum_{s=1}^\alpha 2^{s-1})=\frac{2^\alpha}{n}\sum_{d_1|\ell} \mu(d_1).\]
The last sum is nonzero if and only if $\ell=\frac{n}{2^\alpha}=1.$
Equation~\eqref{Liesup2-3} now follows immediately by invoking the meta-result of Proposition~\ref{metaf}, which says that 
$$Lie_{2n+1}^{(2)}(-1)=-Lie_{2n+1}^{(2)}(1),\quad 
Lie_{2n}^{(2)}(-1)=Lie_{n}^{(2)}(1)-Lie_{2n}^{(2)}(1),$$ 
(or directly by a  more cumbersome case-by-case calculation).
\end{proof}
\begin{thm}\label{PlInv-E} Let $Lie^{(2)}_n$ be the Frobenius characteristic of the induced representation $\exp(\frac{2i\pi}{n}\cdot 2^k)\large\uparrow_{C_n}^{S_n},$  where $k$ is the largest power of 2 which divides $n.$ Then we have the following generating functions:
\begin{enumerate}[itemsep=8pt]
\item \label{ExtLieSup2Reg}\qquad (Exterior powers) \qquad\qquad\qquad\qquad
$ E\left[\sum_{n\geq 1}  Lie^{(2)}_n\right ](t) = (1-tp_1)^{-1}.$

\text{(Alternating symmetric powers)}
\qquad\quad
$H^{\pm}\left[\sum_{n\geq 1} Lie^{(2)}_n\right](t)=1-tp_1.$

That is,  $\sum_{n\geq 1} Lie^{(2)}_n$ is the plethystic inverse of $\sum_{n\geq 1} (-1)^{n-1} h_n.$ 

\item\label{PlInv-ELieSup2}$E\left[\sum_{n\geq 1} (-1)^{n-1}\omega(Lie^{(2)}_n)\right](t)=1+tp_1.$

That is, the  plethystic inverse of $\sum_{n\geq 1} e_n$ is given by 
\begin{center}$\sum_{n\geq 1} (-1)^{n-1}\omega(Lie^{(2)}_n).$
\end{center}

\item\label{SymLieSup2} (Symmetric powers and higher $Lie_n^{(2)}$-modules) 

$H\left[\sum_{n\geq 1}  Lie^{(2)}_n\right](t)
=\prod_{n=2^k, k\geq 0} (1-t^n p_n)^{-1}
=\sum_{{\lambda\in Par}\atop{\text{every part is a power of } 2}} t^{|\lambda|}p_\lambda.$

\item\label{SymAltLieSup2} (Alternating Exterior Powers)

\noindent
$\sum_{\lambda\in Par} (-1)^{|\lambda|-\ell(\lambda)} \omega(E_\lambda[Lie^{(2)}]) (t) =H\left[\sum_{n\geq 1} (-1)^{n-1}\omega(Lie^{(2)}_n)\right](t)$

$=\prod_{n=2^k, k\geq 0} (1+t^n p_n) =\sum_{{\lambda\in D\!Par}\atop{ \text{every part is a power of } 2}}t^{|\lambda|} p_\lambda.$
\end{enumerate}
\end{thm}
\begin{proof}  We apply the meta theorem Theorem~\ref{metathm} to the sequence of symmetric functions $f_n=Lie_n^{(2)}$. All of these identities follow immediately thanks to the values of $Lie_n^{(2)}(t)$ at $t=\pm 1$ given by the preceding lemma.
For the equivalence of the second equation in (1) and the equation (2), we invoke Lemma \ref{pmalt},  which also gives the equivalence of (3) and (4).  
\end{proof}
\noindent
\textit{Remark.}
Observe that Part (3) of the above also gives the description of $Lie_n^{(2)}$ mentioned in the Introduction: If we form the row sums in the character table of $S_n$ corresponding to conjugacy classes of type $\lambda$, where each part of $\lambda$ is a power of 2, then those row sums are nonnegative and produce the representation obtained by symmetrising $Lie_n^{(2)}.$
The theorems about the variant $Lie_n^{(2)} $ in Section 2 now follow easily. 
\vskip .1in
\noindent
\noindent
{\bf Proof of Theorem \ref{Compare}:}

\begin{proof} 
The $Lie^{(2)}$  identities are all restatements of Theorem \ref{PlInv-E} above, using the definition of $H_\lambda$ and $E_\lambda.$  Likewise the $Lie$ identities are all restatements of known results, seeTheorem \ref{ThrallPBWCadoganSolomon} and Theorem~\ref{EquivPBW}.
The first equation in (\ref{TotalCoh}) is, for instance,  a restatement of the first equation of Proposition \ref{ExtLieConj}. 
The statement about the equivalence of the $Lie$ (respectively, $Lie^{(2)}$) identities is a consequence of  Proposition \ref{metaequiv}.
\end{proof}
\vskip .1in
\noindent
{\bf Proof of Theorem \ref{EquivLieSup2}:}
\begin{proof}  Equation (\ref{ExtLieSup2}) is the identity  $E[Lie^{(2)}]=(1-p_1)^{-1}$ of Part (1) of Theorem~\ref{PlInv-E}, and hence by Lemma \ref{pm}, we obtain $H^{\pm}[Lie^{(2)}]=1-p_1,$ which is equation (\ref{PlInvLieSup2}) after removing the constant term and adjusting signs.

Now invoke (\ref{metage2d}) of Theorem \ref{metage2}, Section~\ref{SecMetaThms}.  We have 
$$H^{\pm}[Lie_{\geq 2}^{(2)} ]=E\cdot (1-p_1)
=1+ \sum_{n\geq 2} (e_n-e_{n-1} p_1)
= 1-\omega(\kappa), $$
and this is precisely equation (\ref{Extge2}), again after cancelling the constant term and adjusting signs.  The remaining statements follow exactly as in Theorem \ref{EquivPBW}, Section~\ref{SecMetaThms}.
\end{proof}
\vskip.1in
\noindent 
{\bf Proof of Proposition \ref{AltHLieAltELieSup2}:}
\begin{proof}
Parts (1) and (2) are respectively restatements of Part (2) of Proposition \ref{ExtLieConj}, and  Part (4) of Theorem \ref{PlInv-E}.
\end{proof}

\noindent

The following observation allows us to compute the character values of  $Lie_n^{(2)}$ directly from those of $Lie_n.$  

\begin{thm}\label{LietoLiesup2}  $Lie_n^{(2)}$ is the degree $n$ term in the plethysm
$\sum_{k\geq 0} Lie[p_{2^k}],$ and $Lie_n$ is the degree $n$ term in $Lie^{(2)}-Lie^{(2)}[p_2].$  In particular 
$Lie_n^{(2)}=Lie_n$ if $n$ is odd, and coincides with the sign tensored with $Lie_n$ if $n$ is twice an odd number.
\end{thm}
\begin{proof}  From \eqref{defHE}, it is easy to see that $E=H[p_1-p_2].$ 
By Theorem~\ref{Compare}, we have 
$H[Lie]=E[Lie^{(2)}]$. Putting these two facts together and using associativity of plethysm immediately gives 
$H[Lie]=H[(p_1-p_2)[Lie^{(2)}]],$
and hence, since power sums commute with plethysm,
\[(H-1)[Lie]=(H-1)[Lie^{(2)}-Lie^{(2)}[p_2]].\]
But $H-1$ is invertible with respect to plethysm (see Cadogan's theorem in Theorem~\ref{ThrallPBWCadoganSolomon}), so the result follows. 
It is easy to check that $p_1-p_2$ has plethystic inverse 
$\sum_{k\ge 1} p_{2^k}, $ completing the proof of the first part.  (It is also possible to prove this directly using the Frobenius characteristics, although the computation with Ramanujan sums is somewhat involved.)

The last statement is clear if $n$ is odd. Now suppose $n=2(2m-1)$ for $m\ge 1.$ Then we have $Lie_n^{(2)}= Lie_n+Lie_{2m-1} [p_2].$ 
A routine calculation shows that this coincides with $\omega(Lie_{2(2m-1)})$ (see \cite{GRR}). 
\end{proof}

This yields  the following curious $S_n$-module isomorphism,  giving a recursive definition of $Lie_n^{(2)}:$
\begin{prop}\label{FindingLieSup2}  When $n$ is even:
$$ Lie_n\oplus {\mathbf 1}_{S_2}[Lie_{\frac{n}{2}}^{(2)}]\uparrow_{S_2[S_{\frac{n}{2}}]}^{S_n}\ 
\simeq\,  Lie_n^{(2)} \oplus { \bf sgn}_{S_2}[Lie_{\frac{n}{2}}^{(2)}]\uparrow_{S_2[S_{\frac{n}{2}}]}^{S_n},$$ 
where $S_2[S_{\frac{n}{2}}]$ is the wreath product of $S_2$ with $S_{\frac{n}{2}}$  (i.e. the normaliser of $S_{\frac{n}{2}}\times S_{\frac{n}{2}}$).
If $n$ is odd, this identity simply reduces to the known fact that 
$Lie_n$ and $Lie_n^{(2)}$ coincide.
\end{prop}

The module $Lie_n^{(2)}$ makes an appearance in the decomposition of the module $C\!onj_n$ of the conjugacy action on the class of $n$-cycles as well.   Again we have the following contrasting results between $Lie$ and $Lie^{(2)}$. 
\begin{thm}\label{ConjLieSup2}
\begin{equation}\label{ConjLiepowsums}\sum_n C\!onj_n=\sum_{k\geq 1}p_{k}[Lie];  \text{ equivalently, }
Lie=\sum_{k\geq 1} \mu(k) p_{k}[C\!onj]. \end{equation} 
\begin{equation}\label{ConjLieSup2oddpowsums}\sum_n C\!onj_n=\sum_{k\geq 1}p_{2k-1}[Lie^{(2)}];  \text{ equivalently, } Lie^{(2)}=\sum_{k\geq 1} \mu(2k-1) p_{2k-1}[C\!onj].\end{equation} 
\end{thm}
\begin{proof} The equivalence of the two statements in each case follows by using the fact (easily verified by direct computation) that 
$\sum_{k\ge 1} p_k$ and $\sum_{k\ge 1} \mu(k)p_k$ are plethystic inverses, as are $\sum_{k\ge 1} p_{2k-1}$ and $\sum_{k\ge 1} \mu(2k-1)p_{2k-1}.$
For the first statement, we combine the theorems of Solomon and Thrall in Theorem~\ref{ThrallPBWCadoganSolomon} as follows:
\[H[C\!onj]=\prod_{n\ge 1} p_n[(1-p_1)^{-1}]=\prod_{n\ge 1} p_n[H[Lie]]=\prod_{n\ge 1} H[Lie[p_n]]= H[\sum_{n\ge 1}Lie[p_n]].\]
Here we have used the fact that plethysm is associative, and the commutative property $p_n[f]=f[p_n]$ for power sums.
Hence $(H-1)[C\!onj]=(H-1)[\sum_{n\ge 1}Lie[p_n]].$ Now the result follows as in Theorem~\ref{LietoLiesup2},
which also gives the second part, since it 
says that 
$Lie=(p_1-p_2)[Lie^{(2)}].$  Clearly $\sum_{k\ge 1}p_k[p_1-p_2]
=\sum_{k\ge 1}p_{2k-1}.$
\end{proof}

\section{Meta theorems }\label{SecMetaThms}
In this section we review the meta theorem of \cite{Su1} giving formulas for symmetric and exterior powers of modules induced from centralisers, 
and also further develop these tools in a general setting. Theorem~\ref{metathm} below has wide-ranging applications, as shown in 
\cite{SuFPSAC2018} and \cite{SuFPSAC2019}.
We begin by recalling the following results regarding the sequence of symmetric functions $f_n$ defined in equation (\ref{definef_n}) of Section 1.
\begin{prop}\label{Su1Prop3.1}\cite[Proposition 3.1]{Su1}  Define $F(t)=\sum_{i\geq 1} t^i f_i,$   and define 
$(\omega F)^{alt}(t)=\sum_{i\geq 1} (-1)^{i-1} t^i\omega(f_i).  $  Then 
\begin{align} &F(t)= \log \prod_{d\geq 1} (1-t^d p_d)^{-\frac{\psi(d)}{d}}.\\ 
&(\omega F)^{alt}(t)= \log \prod_{d\geq 1} (1+t^d p_d)^{\frac{\psi(d)}{d}}.\end{align}
\end{prop}

\begin{thm}\label{metathm} \cite[Theorem 3.2]{Su1} Let $F=\sum_{n\geq 1}  f_n$ where $f_n$ is of the form (\ref{definef_n}), $H(v)=\sum_{n\geq 0} v^n h_n$ and 
$E(v)=\sum_{n\geq 0}  v^n e_n.$  
We have the following plethystic generating functions:

\noindent
 (Symmetric powers) 
\begin{equation}\label{metaSym}H(v)[F] = \sum_{\lambda\in Par} v^{\ell(\lambda)}  H_\lambda[F]=\prod_{m\geq 1} (1-p_m)^ {-f_m(v)}.\end{equation}

(Exterior powers) 
\begin{equation}\label{metaExt}E(v)[F]=\sum_{\lambda\in Par} v^{\ell(\lambda)} E_\lambda[F] =\prod_{m\geq 1} (1- p_m)^{f_m(-v)}.\end{equation}

 (Alternating exterior powers)
\begin{center}$\sum_{\lambda\in Par}  (-1)^{|\lambda|-\ell(\lambda)}v^{\ell(\lambda)} \omega(E_\lambda[F])$\end{center}
\begin{equation}\label{metaAltExt}=\sum_{\lambda \in Par} v^{\ell(\lambda)}H_\lambda[\omega(F)^{alt}]
=H(v)[\omega(F)^{alt}]
= \prod_{m\geq 1} (1+ p_m)^{f_m(v)}.\end{equation}

(Alternating symmetric powers) 
\begin{center}$\sum_{\lambda\in Par} (-1)^{|\lambda|-\ell(\lambda)}v^{\ell(\lambda)} \omega(H_\lambda[F])$\end{center}
\begin{equation}\label{metaAltSym}=\sum_{\lambda \in Par} v^{\ell(\lambda)}E_\lambda[\omega(F)^{alt}]
=E(v)[\omega(F)^{alt}]
= \prod_{m\geq 1} (1+ p_m)^{-f_m(-v)}.\end{equation}
\end{thm}

\begin{prop}\label{metaf} \cite[Lemma 3.3]{Su1} The numbers $f_n(1)$ and $f_n(-1)$ determine each other according to the equations 
$f_{2m+1}(-1)=-f_{2m+1}(1)$ for all $m\geq 0,$ and 
$f_{2m}(-1)=f_m(1) -f_{2m}(1)$ for all $m\geq 1.$  In fact, 
the symmetric functions 
$f_n = \dfrac{1}{n} \sum_{d|n} \psi(d) p_d^{\frac{n}{d}}$
are  determined by the numbers 
 $f_n(1)=\dfrac{1}{n} \sum_{d|n}\psi(d) ,$
or by the numbers $f_n(-1)=\dfrac{1}{n} \sum_{d|n}\psi(d) (-1)^{\frac{n}{d}}.$
\end{prop}

Recall from Section 1.1 that we define  $H^{\pm}=\sum_{r\geq 0} (-1)^r h_r$ and $E^{\pm}=\sum_{r\geq 0} (-1)^r e_r.$ Thus 
$H^{\pm}=1-H^{alt},$ where $H^{alt}=\sum_{r\geq 1} (-1)^{r-1} h_r,$ and likewise $E^{\pm}=1-E^{alt}.$  The following identity is well known (see \cite[(2.6)]{M}, \cite[Section 7.6]{St4EC2}).
\begin{equation}\label{HE}\left(\sum_{n\geq 0} t^n h_n\right) \left( \sum_{n\geq 0} (-t)^n e_n\right) =1. \text{ Equivalently, } H^{\pm}\cdot E=1=H\cdot E^{\pm}.\end{equation}
This identity is generalised in Lemma \ref{pm} below.

\begin{lem}\label{pm} Let $F=\sum_{n\geq 1} f_n,$  $G=1+\sum_{n\geq 1} g_n$ and $K=1+\sum_{n\geq 1} k_n$ be arbitrary formal series of symmetric functions, as usual with $f_n, g_n, k_n$ being of homogeneous degree $n.$
\begin{enumerate}
\item 
$H[F]=G\iff E^{\pm}[F]=\dfrac{1}{G} \iff \sum_{r\geq 1} (-1)^{r-1} e_r[F] =\dfrac{G-1}{G}.$
\item $E[F]=K\iff H^{\pm}[F]=\dfrac{1}{K} \iff \sum_{r\geq 1} (-1)^{r-1} h_r[F] =\dfrac{K-1}{K}.$
\end{enumerate}
\end{lem}
\begin{proof}\begin{enumerate}
\item By definition,  $E^{\pm}=\sum_{r\geq 0} (-1)^r e_r=1/H,$ 
and hence the first equality follows. For the second equality, note that $\sum_{r\geq 1} (-1)^{r-1} e_r= 1-E^{\pm}=1 -1/H,$ and hence 
$\sum_{r\geq 1} (-1)^{r-1} e_r[F]= 1 -1/H[F]$ as claimed. The reverse direction is clear.
\item This follows exactly as above, since $H^{\pm}=\sum_{r\geq 0} (-1)^r h_r=1/E.$
\end{enumerate}
\end{proof}

\begin{lem}\label{pmalt} Let $G=\sum_{n\geq 1} g_n,$ $K=\sum_{n\geq 0} k_n,$ where $g_n, k_n$ are symmetric functions of homogeneous degree $n$ for $n\geq 1,$ and $k_0=1.$   Let $K^\pm$ denote the sum 
$\sum_{n\geq 0} (-1)^n k_n.$ Then 
\begin{align} 
H[\sum_{n\geq 1} (-1)^{n-1}\omega(g_n)]=K&\iff 
H[\sum_{n\geq 1} g_n] = \dfrac{1}{K[-p_1]}=\dfrac{1}{\omega(K)^\pm}\\
&\iff E^{\pm}[\sum_{n\geq 1} g_n] = \omega(K)^\pm.
\end{align}
\begin{align} 
E[\sum_{n\geq 1} (-1)^{n-1}\omega(g_n)]=K &\iff 
E[\sum_{n\geq 1} g_n] =  \dfrac{1}{K[-p_1]}=\dfrac{1}{\omega(K)^\pm}\\
&\iff H^{\pm}[\sum_{n\geq 1} g_n] = \omega(K)^\pm.
\end{align}

\end{lem}
\begin{proof} We use the fact that for any symmetric functions $f_1, f_2$ of homogeneous degree, $f_1[-f_2]=(-1)^{\mathrm{deg} f_1} \omega(f_1)[f_2].$ In particular this implies $K[-p_1]=\omega(K)^{\pm}$ whenever $K$ is a series of symmetric functions $k_n$ of homogeneous degree $n.$
Hence, using associativity of plethysm,
$$K=H[-\sum_{n\geq 1} (-1)^n  \omega(g_n)]=H[-G[-p_1]]
=(H[-G])[-p_1],$$
$$\text{or equivalently }\quad K[-p_1]=E^{\pm}[G]=(\frac{1}{H})[G]=\dfrac{1}{H[G]},$$
$$\text{ and finally }\quad H[G]=\dfrac{1}{K[-p_1]}=\dfrac{1}{\omega(K)^{\pm}}.$$
The equivalence of the first two equations  is a consequence of the fact that $H[G]=(\dfrac{1}{E^{\pm}})[G] =\dfrac{1}{E^{\pm}[G]}.$
The   equivalences of the second pair follow in a similar manner. 
\end{proof}

Lemma (\ref{pmalt}) explains, in greater generality,  the connection between equations (\ref{metaSym}) and (\ref{metaAltExt}) (resp. (\ref{metaExt}) and (\ref{metaAltSym})).  In fact these lemmas give us  the following observation. Let $F,$ $H,$ and $E$ be as defined in Theorem \ref{metathm}.  Then

\begin{prop}\label{metaequiv}  The identities of Theorem \ref{metathm} are all equivalent, and are also equivalent to 
\begin{equation}E^{\pm}(v)[F] 
=\prod_{m\geq 1} (1- p_m)^ {f_m(v)}.\end{equation}

\begin{equation}H^{\pm}(v)[F]=\prod_{m\geq 1} (1- p_m)^{-f_m(-v)}.\end{equation}
\end{prop}

Now let $F=\sum_{n\geq 1} f_n$ be an arbitrary series of symmetric functions $f_n$ homogeneous of degree $n.$   In particular $f_n$ need not be of the form (\ref{definef_n}). We write $F_{\geq 2}$ for the series $\sum_{n\geq 2} f_n.$  

\begin{thm}\label{metage2}\cite[Proposition 2.3, Corollary 2.4]{Su1} Assume that $F=\sum_{n\geq 1} f_n$ is any series of symmetric functions $f_n$ homogeneous of degree $n.$ Also assume $f_1=p_1.$  Then we have the following  identities:
 \begin{equation}\label{metage2a}
H(v)[F_{\geq 2}]
=E(-v) \cdot H(v)[F].
\end{equation}
Equivalently,
\begin{equation}\label{metage2b}
E^{\pm}(v)[F_{\geq 2}]=E^{\pm}(v)[\sum_{n\geq 2} f_n] 
=\dfrac{H(v)}{ H(v)[F]}.
\end{equation}
\begin{equation}\label{metage2c}
E(v)[F_{\geq 2}]
=H(-v) \cdot E(v)[F].
\end{equation}
Equivalently,
\begin{equation}\label{metage2d}
H^{\pm}(v)[F_{\geq 2}]=H^{\pm}(v)[\sum_{n\geq 2} f_n] 
=\dfrac{E(v)}{ E(v)[F]}.
\end{equation}

\end{thm}

\begin{proof} The equivalence of (\ref{metage2a}) and (\ref{metage2b}) follows because of the identity (\ref{HE}).  Consider now equation (\ref{metage2a}). By standard properties of the skewing operation and the plethysm operation (see, e.g. 
\cite[(8.8)]{M}), we know that 
$h_n[G_1+G_2]=\sum_{k=0}^n h_k[G_1] h_{n-k}[G_2].$
This in turn gives $$H[G_1+G_2]= H[G_1] \, H[G_2].$$  
Taking $G_1=f_1$ and $G_2=F- f_1,$ we have
$$ H[F] =H[f_1]\, H[\sum_{n\geq 2} f_n ].$$
But $H[ f_1]=H[ p_1]=\sum_{n\geq 0} h_n.$   
Hence, using (\ref{HE}), $$\dfrac{1}{H(v)[ f_1]}= \sum_{n\geq 0} (-v)^n e_n.$$
The equations (\ref{metage2c}) and (\ref{metage2d}) are obtained in  entirely analogous fashion.
\end{proof}

An important consequence of Theorem \ref{metage2} is worth pointing out.  Denote by $Lie_n$ the Frobenius characteristic of the $S_n$-representation afforded by the multilinear component of the free Lie algebra on $n$ generators.  Let $Lie=\sum_{n\geq 1} Lie_n.$  
This  special case of equation (\ref{metaSym}),  Theorem \ref{metathm}, obtained by taking $\psi(d)=\mu(d)$ in \eqref{definef_n}, and hence $f_n=Lie_n,$ yields
\begin{equation*} H[Lie]=(1-p_1)^{-1} \end{equation*}
This is Thrall's theorem.  See 
Theorem \ref{ThrallPBWCadoganSolomon} and more generally\cite{R}. 
Define a symmetric function $\kappa=\sum_{n\geq 2} s_{(n-1,1)},$ where $s_{(n-1,1)}$ is the Schur function indexed by the partition $(n-1,1).$ (This is the Frobenius characteristic of the standard  representation of $S_n.$) 

Lemma \ref{pm} and Theorem \ref{metage2} now imply that  Thrall's theorem is in fact {\it equivalent} to the derived series decomposition of the free Lie algebra \cite{R}.  More precisely, the following identities are equivalent:
\begin{thm}\label{EquivPBW}  (Equivalence of Thrall's theorem and derived series for free Lie algebra)
\begin{equation}\label{EquivPBWSym}  (H-1)[Lie]=(\sum_{r\geq 1} h_r)[Lie]=\sum_{n\geq 1} p_1^n. \end{equation}
\begin{equation}\label{PlInv-HCadogan}  (H-1)[\sum_{i\geq 1} (-1)^{i-1}\omega(Lie_i)]=p_1,
\end{equation}
(the plethystic inverse of the sum of homogeneous symmetric functions $\sum_{r\geq 1} h_r$).
\begin{equation}\label{EquivPBWAltExt}(1-E^{\pm})[Lie]= (\sum_{r\geq 1} (-1)^{r-1} e_r)[Lie]=p_1,\end{equation}
(the plethystic inverse of the sum $\sum_{r\geq 1} Lie_r$).
\begin{equation}\label{EquivPBWAltExtge2a}(1-E^{\pm})[Lie_{\geq 2}] =(\sum_{r\geq 1} (-1)^{r-1} e_r)[Lie_{\geq 2}]=\kappa. \end{equation}
\begin{equation}\label{EquivPBWAltExtge2b}  \text{ The degree $n$ term in }\sum_{r\geq 0} (-1)^{n-r} e_{n-r}[Lie_{\geq 2}] \text{ is }(-1)^{n-1}  s_{(n-1,1)},
\end{equation}
\begin{equation}\label{LieFilta}   Lie_{\geq 2}=Lie[\kappa],\end{equation}
\begin{equation} \label{LieFiltb}
 Lie_{\geq 2}=\kappa+\kappa[\kappa]+\kappa[\kappa[\kappa]]+\ldots  
  \end{equation}
  (The derived series filtration of the free Lie algebra)
 \begin{equation}\label{LieHodge} (H-1)[Lie_{\geq 2}]=(\sum_{r\geq 1} h_r)[Lie_{\geq 2}]=(1-p_1)^{-1}\cdot E^{\pm}-1
 =\sum_{n\geq 2}\sum_{k=0}^n (-1)^k p_1^{n-k}e_k. \end{equation}
\end{thm}

\begin{proof} We specialise the preceding identities to $v=1.$  Equation (\ref{EquivPBWSym}) is equivalent to $H[Lie]=(1-p_1)^{-1},$ and hence by Lemma \ref{pm}, we obtain $E^{\pm}[Lie]=1-p_1,$ which is equation (\ref{EquivPBWAltExt}) after removing the constant term and adjusting signs.

Now invoke (\ref{metage2b}) of Theorem \ref{metage2}.  We have 
$$E^{\pm}[Lie_{\geq 2} ]=H\cdot (1-p_1)
=1+ \sum_{n\geq 2} (h_n-h_{n-1} p_1)
= 1-\kappa, $$
and this is precisely equation (\ref{EquivPBWAltExtge2a}), again after cancelling the constant term and adjusting signs.  Since these steps are clearly reversible, we see that (\ref{EquivPBWAltExt}) and (\ref{EquivPBWAltExtge2a}) are in fact equivalent.

The equivalence of (\ref{EquivPBWAltExtge2a}) and (\ref{EquivPBWAltExtge2b})-(\ref{LieFilta}) follows by applying the plethystic inverse of $\sum_{r\geq 1} (-1)^{r-1}e_r, $ which is given by (\ref{EquivPBWAltExt}). 

It is clear by iteration that (\ref{LieFilta}) gives (\ref{LieFiltb}).  In the reverse direction, we can rewrite (\ref{LieFiltb}) as 
\begin{center}
$Lie=p_1+\kappa+\kappa[\kappa]+\kappa[\kappa[\kappa]]+\ldots,$\end{center} 
and hence 
$Lie[\kappa]=Lie-p_1=Lie_{\geq 2},$
which is (\ref{LieFilta}).
Finally the equivalence of (\ref{EquivPBWSym}) and ((\ref{LieHodge})) follows from equation (\ref{metage2a}) of Theorem \ref{metage2} and equation (\ref{HE}).  
\end{proof} 

\begin{thm}\label{Restrictge2}  Let $F,H,E$ be as in Theorem \ref{metathm}, and assume $f_1=p_1.$ 
\begin{enumerate}
\item Let $\prod_{m\geq 2} (1-p_m)^{-f_m(1)}=\sum_{n\geq 0} g_n$ for homogeneous symmetric functions $g_n$ of degree $n,$ $g_0=1.$  (Note that $g_1=0.$)
Also define $\sigma_n=\sum_{i\geq 0} (-1)^i e_{n-i} g_i.$ Then $\sigma_n, n\geq 1,$ is the characteristic of a one-dimensional, possibly  virtual representation, with the property that its restriction to $S_{n-1}$ is $\sigma_{n-1}.$   Let $\alpha_n$ be the degree $n$ term in $H[F_{\geq 2}], n\geq 0.$ (Note that $\alpha_0=1$ and $\alpha_1=0.$ )  Then we have the recurrence
\begin{equation} 
\alpha_n=p_1\alpha_{n-1}+(-1)^n\sigma_n.
\end{equation}
\item  Let $\prod_{m\geq 2} (1-p_m)^{f_m(-1)}=\sum_{n\geq 0} k_n$ for homogeneous symmetric functions $k_n$ of degree $n,$ $k_0=1.$  (Note that $k_1=0.$)
Also define $\tau_n=\sum_{i\geq 0} (-1)^i h_{n-i} k_i.$ Then $\tau_n, n\geq 1, $ is the characteristic of a one-dimensional, possibly  virtual representation, with the property that its restriction to $S_{n-1}$ is $\tau_{n-1}.$ Let $\beta_n$ be the degree $n$ term in $E[F_{\geq 2}], n\geq 0.$ (Note that $\beta_0=1$ and $\beta_1=0.$ )  Then we have the recurrence
\begin{equation} 
\beta_n=p_1\beta_{n-1}+(-1)^n\tau_n.
\end{equation}
\end{enumerate}
\end{thm}

\begin{proof}  \begin{enumerate} 
\item The hypothesis that the degree one term $f_1$ in $F$ equals $p_1$ implies that $f_1(1)=1.$ 
 From equation (\ref{metaSym}) of Theorem \ref{metathm} and (\ref{metage2a}) of Theorem \ref{metage2}, we now have 
\begin{align*}
(1-p_1) H[F_{\geq 2}]&= E^{\pm} \cdot \prod_{m\geq 2} (1-p_m)^{-f_m(1)}\\
&=\sum_{r\geq 0} (-1)^r e_r\sum_{n\geq 0} g_n=\sum_{n\geq 0} \sum_{i=0}^n (-1)^{n-i} e_{n-i} g_i
=\sum_{n\geq 0} (-1)^n \sigma_n,
\end{align*}
from which the recurrence is clear.  It remains to establish the statement about $\tau_n.$ First observe that since $p_1$ does not appear in the power-sum expansion of $g_n,$ for $i\geq 1,$ $e_{n-i}g_i$ is the Frobenius characteristic of a zero-dimensional (hence virtual) representation (dimension is computed, for example,  by taking the scalar product with $p_1^n$).  The dimension of $\sigma_n$ is therefore that of $e_n,$ and is thus one.  To verify the statement about the restriction, we use the fact that the Frobenius characteristic of the restriction is the partial derivative $\dfrac{\partial \sigma}{\partial p_1}.$ The partial derivative of $e_n$ with respect to $p_1$ is clearly $e_{n-1}, n\geq 1,$ and that of $g_n$ with respect to $p_1$ is clearly 0.  The claim follows.

\item Again, $f_1=p_1$ implies $1=-f_1(-1).$ The argument is identical, but now use equation (\ref{metaExt}) and equation (\ref{metage2c}).
\end{enumerate}
\end{proof}

Note that, with $F$ as in Theorem \ref{metathm},  the dimension of the representation whose characteristic is $h_j[F]|_{{\rm deg\ }n} $ (respectively $e_j[F]|_{{\rm deg\ }n} $) is the number $c(n,j)$ of permutations in $S_n$ with $j$ disjoint cycles. 
Similarly  the dimension of the representation whose characteristic is $h_j[F_{\geq 2}]|_{{\rm deg\ }n} $ 
(respectively $e_j[F_{\geq 2}]|_{{\rm deg\ }n} $) is the number $d(n,j)$ of fixed-point-free permutations, or derangements, in $S_n$ with $j$ disjoint cycles, and hence the dimension of $\alpha_n $ (respectively $\beta_n$) is the total number of derangements $d_n.$ Hence, taking dimensions in either of the above recurrences, we recover the well-known recurrence $d_n=nd_{n-1} +(-1)^n, n\geq 2.$   On the other hand, the recurrence (\ref{Restrict2}) below is the symmetric function analogue of the recurrence 
$$d(n,j)=n(d(n-1,j)+d(n-2, j-1)),$$ while (\ref{Restrict1}) is the analogue of 
$$c(n,j)=c(n-1,j)+n c(n-1, j-1).$$

From Theorem \ref{metage2} we can also deduce interesting recurrences for the restrictions of the symmetric and exterior powers of $F$ from $S_n$ to $S_{n-1},$ for an arbitrary formal sum $F$ of homogeneous symmetric functions $f_n$ having the following key property: each $f_n$ is the Frobenius characteristic of an $S_n$-representation (possibly virtual) which restricts to the regular representation of $S_{n-1}.$ See also \cite[Proposition 3.5]{Su0}.  

\begin{thm}\label{Restrict}  Let $F=\sum_{n\geq 1} f_n.$  Assume $f_n$ is a symmetric function of homogeneous degree $n$ with the following property: $\dfrac{\partial}{\partial p_1} f_n=p_1^{n-1}, n\geq 1.$ 
\begin{enumerate}
\item Let ${G}^j_n$ equal either $h_j[F]|_{{\rm deg\ }n}$  or  $e_j[F]|_{{\rm deg\ }n}.$ Then for $n\geq 1$ and $0\leq j\leq n$ we have
\begin{equation} \label{Restrict1}
\dfrac{\partial}{\partial p_1} {G}^{n-j}_{n}={G}^{n-1-j}_{n-1}+p_1\dfrac{\partial}{\partial p_1} {G}^{n-1-(j-1)}_{n-1}. 
\end{equation}
\item 
 Let $\hat{G}^j_n$ equal either $h_v[F_{\geq 2}]|_{{\rm deg\ }n}$   or $e_j[F_{\geq 2}]|_{{\rm deg\ }n}.$ Then for $n\geq 2$ and $1\leq j\leq n-1,$  we have
\begin{equation}\label{Restrict2}
\dfrac{\partial}{\partial p_1} \hat{G}^{n-j}_n=p_1(\dfrac{\partial}{\partial p_1} \hat{G}^{(n-1)-(j-1)}_{n-1} +\hat{G}^{n-2-(j-1)}_{n-2}).
\end{equation}
\end{enumerate}
\end{thm}

\begin{proof}  The hypothesis about the $f_n$ implies that  derivative of $F$ with respect to $p_1$ is $\sum_{n\geq 1} p_1^{n-1}=(1-p_1)^{-1}.$ Also note that 
$$\dfrac{\partial}{\partial p_1}H(v)=vH(v), \qquad
\dfrac{\partial}{\partial p_1}E(v)=vE(v).$$
\begin{enumerate}
\item  The chain rule gives 
\begin{equation*}\dfrac{\partial}{\partial p_1}\left(H(v)[F]\right)=v\cdot H(v)[F]\cdot (1-p_1)^{-1}\Longrightarrow
v\cdot H(v)[F]=(1-p_1)\dfrac{\partial}{\partial p_1}\left(H(v)[F]\right);\qquad (A)\end{equation*}
 \begin{equation*}v\cdot E(v)[F]=(1-p_1)\dfrac{\partial}{\partial p_1}\left(E(v)[F]\right).\qquad(B) \end{equation*}
If ${G}^j_n=h_j[F]|_{{\rm deg\ }n}, $ then 
$H(v)[F]=\sum_{j\geq 0}\sum_{n\geq 0} G^j_n.$
The result follows by extracting the symmetric function of degree $n-1$ on each side of $(A)$, and the coefficient of $v^{n-j}.$ 
The  recurrence for $e_j[F]$ is identical in view of equation $(B)$; now use the expansion $E(v)[F]=\sum_{j\geq 0}\sum_{n\geq 0} G^j_n.$

\item  Now we use the identity (\ref{metage2a}) of Theorem \ref{metage2}.  We have 
$$H(v)[F_{\geq 2}]=E^{\pm}(v)\cdot H(v)[F].$$
Using the fact that $$\dfrac{\partial}{\partial p_1}E^{\pm}(v)=-v\cdot E^{\pm}(v),$$ and the chain rule, we obtain
\begin{align*} 
\dfrac{\partial}{\partial p_1} (H(v)[F_{\geq 2}]) &=-v\cdot E^{\pm}(v) \cdot H(v)[F_{\geq 2}]
+E^{\pm}\cdot \dfrac{\partial}{\partial p_1} (H(v)[F_{\geq 2}])\\
&= -v\cdot H(v)[F_{\geq 2}] +E^{\pm}\cdot v\cdot H(v)[F]\cdot (1-p_1)^{-1},
\end{align*}
where we have used the computation in (1).  It follows that 
$$(1-p_1)\dfrac{\partial}{\partial p_1} (H(v)[F_{\geq 2}])
= v\cdot p_1 (H(v)[F_{\geq 2}]), \text{ and hence}$$
\begin{equation*}\dfrac{\partial}{\partial p_1} (H(v)[F_{\geq 2}])
=p_1\left( \dfrac{\partial}{\partial p_1} (H(v)[F_{\geq 2}])+
v\cdot \dfrac{\partial}{\partial p_1} (H(v)[F_{\geq 2}])\right).\end{equation*}
 Let $H(v)[F_{\geq 2}]= \sum_{i\geq 0} \sum_{n\geq 1} v^j \hat{G}^j_n.$
The recurrence follows by extracting the symmetric function of degree $n-1$ on each side, and the coefficient of $v^{n-j}.$  
Similarly, in view of the identity (\ref{metage2c}) of Theorem \ref{metage2} and the fact that $$\dfrac{\partial}{\partial p_1}H^{\pm}(v)=-v\cdot H^{\pm}(v),$$
we obtain 
\begin{equation*}\dfrac{\partial}{\partial p_1} (E(v)[F_{\geq 2}])
=p_1\left( \dfrac{\partial}{\partial p_1} (E(v)[F_{\geq 2}])+
v\cdot \dfrac{\partial}{\partial p_1} (E(v)[F_{\geq 2}])\right).\end{equation*}
Hence it is clear that the same recurrence holds for $e_j[F_{\geq 2}].$
\end{enumerate}
\end{proof}

In particular, Theorem \ref{Restrict} applies to the family of representations whose characteristic $f_n$ is defined by equation (\ref{definef_n}), provided $\psi(1)=1.$ The latter condition guarantees that each $f_n$ restricts to the regular representation of $S_{n-1}.$ 

In the recent paper \cite{HR}, Hersh and Reiner derive several identities and recurrences for what are essentially the symmetric and exterior powers of $Lie.$  The connection between the work of \cite{HR} and the specialisation of our results to $F=Lie$, is the well-known fact ( see \cite{LS}, \cite{Su1}) that the exterior power of $Lie,$ when tensored with the sign representation, coincides with the Whitney homology of the lattice of set partitions.  Here we record the conclusions for the special setting of $F=Lie$.   In the notation of \cite{HR}, we have ${\rm ch\ }\widehat{Lie}^i_n=h_{n-i}[Lie_{\geq 2}]\vert_{{\rm deg\ }n},$ 
while ${\rm ch\ }\omega(\widehat{W}^i_n)=e_{n-i}[Lie_{\geq 2}]\vert_{{\rm deg\ }n}.$  Also let $\widehat{Lie}_n=\sum_{i\geq 0}^{n-1}\widehat{Lie}^i_n,$ and $\widehat{W}_n=\sum_{i\geq 0}^{n-1}\widehat{W}^i_n.$

\begin{cor}\label{HRRepStab} (The case $F=Lie$)
\begin{enumerate}
\item (\cite[Theorem 1.7]{HR})  
$\sum_{i\geq 0}(-1)^i {\rm ch\ }\omega(\widehat{W}^i_n)
=(-1)^{n-1} s_{(2, 1^{n-2})}.$
\item (\cite[Theorem 1.2]{HR})  

 $ {\rm ch\ }\widehat{Lie}_n=
(H-1)[Lie_{\geq 2}]|_{{\rm deg\ }n}= p_1 \cdot  {\rm ch\ }\widehat{Lie}_{n-1}+(-1)^n e_n, 
$

$ {\rm ch\ }\widehat{W}_n=
(E-1)[Lie_{\geq 2}]|_{{\rm deg\ }n}= p_1 \cdot {\rm ch\ }\widehat{W}_{n-1} +(-1)^n \tau_n , $

where $\tau_n= s_{(n-2, 1^2)}-s_{(n-2,2)},n \geq 4,$ and $\tau_3=s_{(n-2, 1^2)}.$
 
 \item  (\cite[Theorem 1.4]{HR}) $$\dfrac{\partial}{\partial p_1} {\rm ch\ }\widehat{Lie}^j_n=p_1(\dfrac{\partial}{\partial p_1} {\rm ch\ }\widehat{Lie}^{(j-1)}_{n-1} +{\rm ch\ }\widehat{Lie}^{(j-1)}_{n-2}),$$
 $$\dfrac{\partial}{\partial p_1} {\rm ch\ }\widehat{W}^j_n=p_1(\dfrac{\partial}{\partial p_1} {\rm ch\ }\widehat{W}^{(j-1)}_{n-1} +{\rm ch\ }\widehat{W}^{(j-1)}_{n-2}).$$
\end{enumerate}
\end{cor}
\begin{proof} Clearly (1) is just equation (\ref{EquivPBWAltExtge2b}) tensored with the sign.

For (2),  apply Theorem \ref{Restrictge2}  to $F=Lie=\sum_{n\geq 1} Lie_n.$  
It is clear that in this case we have $g_n=0, n\geq 2,$ and $k_2=-p_2, k_n=0, n\geq 3.$ Hence Theorem \ref{Restrictge2} gives the following:
  \begin{equation*}
(H-1)[Lie_{\geq 2}]|_{{\rm deg\ }n}= p_1 \cdot (H-1)[Lie_{\geq 2}]|_{{\rm deg\ }n-1} +(-1)^n e_n 
\end{equation*}
and 
\begin{equation*}
(E-1)[Lie_{\geq 2}]|_{{\rm deg\ }n}= p_1 \cdot (E-1)[Lie_{\geq 2}]|_{{\rm deg\ }n-1} +(-1)^n \tau_n , 
\end{equation*}
where $\tau_n=h_n-h_{n-2} p_2 = s_{(n-2, 1^2)}-s_{(n-2,2)},n \geq 4,$ and $\tau_3=s_{(n-2, 1^2)}.$

Part (3) is immediate from Theorem \ref{Restrict}, which applies since it is well known that $Lie_n$ restricts to the regular representation of $S_{n-1}.$  When $F=Lie_n,$ the functions $G_n^j$ become ${\rm ch\ }\widehat{Lie}_n^{n-j}$ when applied to the symmetric powers $H$, and $\omega({\rm ch\ }\widehat{W}_n^{n-j})$  when applied to the exterior powers $E$.
\end{proof}

  Hersh and Reiner use the second recurrence in (2) above  to establish a remarkable formula for the decomposition into irreducibles for  $(E-1)[Lie_{\geq 2}]|_{{\rm deg\ }n},$ in terms of certain standard Young tableaux that they call {\it Whitney-generating tableaux} \cite[Theorem 1.3]{HR}.   The decomposition into irreducibles of $(H-1)[Lie_{\geq 2}]|_{{\rm deg\ }n}$ is similarly given as a sum of {\it desarrangement  tableaux} \cite[Section 7]{HR}. This was established from the first recurrence in (2) above, for the sign-tensored version, in \cite[Proposition 2.3]{RW}. 


\bibliographystyle{amsplain.bst}

\end{document}